\documentclass[10pt]{article}
\usepackage{amsmath,amsfonts,amssymb,graphicx,amsthm,accents,verbatim,xcolor}
\usepackage{doi,mathtools,cite,caption,subcaption,titlefoot,geometry}
\usepackage[T1]{fontenc}
\usepackage[utf8]{inputenc}
\usepackage[english]{babel}
\usepackage [autostyle, english = american]{csquotes}

\usepackage{hyperref}

\definecolor{link1}{rgb}{0,0,.7}
\definecolor{link2}{rgb}{0,0.25,0.5}
\hypersetup{final,colorlinks,allcolors=link1,citecolor=link2}
\MakeOuterQuote{"} 
\newcommand{\defeq}{\stackrel{\scriptscriptstyle\textup{def}}{=}}
\colorlet{darkgreen}{green!40!black}
\providecommand{\email}[1]{\href{mailto:#1}{#1}}
\providecommand{\cites}[1]{\cite{#1}}
\renewcommand{\le}{\leqslant}
\renewcommand{\leq}{\leqslant}
\renewcommand{\ge}{\geqslant}
\renewcommand{\geq}{\geqslant}
\newcommand{\grad}{\nabla}
\renewcommand{\epsilon}{\varepsilon}

\numberwithin{equation}{section}
\theoremstyle{plain}
\newtheorem*{theorem*}{Theorem}
\newtheorem*{lemma*}{Lemma}
\newtheorem{theorem}{Theorem}
\newtheorem{lemma}{Lemma}[section]

\newtheorem{question}[lemma]{Question}
\newtheorem{conj}[lemma]{Conjecture}

\newtheorem{proposition}[lemma]{Proposition}

\theoremstyle{definition}

\theoremstyle{remark}
\newtheorem{remark}[lemma]{Remark}
\newtheorem*{remark*}{Remark}

\DeclarePairedDelimiter{\abs}{\lvert}{\rvert}
\DeclarePairedDelimiter{\norm}{|}{|}
\DeclarePairedDelimiter{\paren}{(}{)}

\DeclarePairedDelimiter{\set}{\{}{\}}

\newcommand{\ve}{\varepsilon}
\newcommand{\rmd}{\, \mathrm{d}}

\newcommand{\be}{\begin{equation}}
\newcommand{\ee}{\end{equation}}
\newcommand{\ba}{\begin{array}{l}}
\newcommand{\ea}{\end{array}}

\newcommand{\bbT}{{\mathbb{T}}}

\newcommand{\eps}{{\varepsilon}}

\newcommand{\rd}{\partial}
\newcommand{\nb}{\nabla}

\newcommand{\T}{\mathbb{T}}
\newcommand{\R}{\mathbb{R}}
\def\ol{\overline}

\begin{document}

\title{%
  Anomalous Dissipation in Passive Scalar Transport
}
\unmarkedfntext{%
  This work has been partially supported by the National Science Foundation under grants
  DMS-1703997 to TD,
  DMS-1817134 to TE,
  DMS-1814147 to GI,
  as well as
  by the Center for Nonlinear Analysis.
  IJ was partially supported by the Science Fellowship of POSCO TJ Park Foundation and the National Research Foundation of Korea (NRF) grant No. 2019R1F1A1058486.
  TD and TE would like to thank the Korean Institute for Advanced Study (KIAS) for its hospitality.  TD would like to thank Navid Constantinou for useful discussions.
}
\author{%
  Theodore D. Drivas%
  \footnote{%
    Department of Mathematics, Princeton University.
    E-mail: \email{tdrivas@math.princeton.edu}
  }
  \and
  Tarek M. Elgindi%
  \footnote{%
    Department of Mathematics, UC San Diego.
    E-mail: \email{telgindi@ucsd.edu}
  }
  \and
  Gautam Iyer%
  \footnote{%
    Department of Mathematical Sciences, Carnegie Mellon University.
    E-mail: \email{gautam@math.cmu.edu}
  }
  \and
  In-Jee Jeong%
  \footnote{%
    School of Mathematics, Korea Institute for Advanced Study.
    E-mail: \email{ijeong@kias.re.kr}}
}
\date{\today}
\maketitle

\begin{abstract}

We study anomalous dissipation in hydrodynamic turbulence in the context of passive scalars.
Our main result produces an incompressible $C^\infty([0,T)\times \mathbb{T}^d)\cap L^1([0,T]; C^{1-}(\mathbb{T}^d))$ velocity field which explicitly exhibits anomalous dissipation.
As a consequence, this example also shows non-uniqueness of solutions to the transport equation with an incompressible $L^1([0,T]; C^{1-}(\mathbb{T}^d))$ drift, which is smooth except at one point in time.
We also provide three sufficient conditions for anomalous dissipation provided solutions to the inviscid equation become singular in a controlled way.
Finally, we discuss connections to the Obukhov-Corrsin monofractal theory of scalar turbulence along with other potential applications.
\end{abstract}


\section{Introduction}

We study the advection-diffusion equation
\begin{equation}\label{visceqn}
\partial_t \theta^\kappa + u \cdot \nabla \theta^\kappa =\kappa \Delta \theta^\kappa \,,
\end{equation}
on the $d$-dimensional torus, $\mathbb{T}^d$.
Here $\theta^\kappa$ is a passive scalar, representing temperature or concentration, $\kappa > 0$ is the molecular diffusivity, and $u$ is a prescribed, time dependent divergence free vector field representing the velocity of an ambient fluid.

Since $u$ is divergence free, one immediately sees that the $L^2$ energy decay of solutions is governed by
\begin{equation}\label{e:ee}
  \frac{1}{2}  \abs{\theta^\kappa(t)}_{L^2}^2  =   \frac{1}{2}  \abs{\theta(0)}_{L^2}^2 - \kappa \int_0^t \abs{\grad \theta^\kappa(s)}_{L^2}^2\rmd s\,,
\end{equation}
and thus the $L^2$ energy dissipation can be measured using $\kappa \int_0^t \abs{\grad \theta^\kappa}_{L^2}^2\rmd s$.
Even though the advecting velocity field doesn't feature in~\eqref{e:ee}, it influences the energy decay indirectly.
Indeed, advection typically generates small scales, which are rapidly damped by the diffusion.
What is expected in certain turbulent regimes~\cite{ShraimanSiggia00,DonzisSreenivasanEA05,Sreenivasan19} is that these effects strike a balance and the energy dissipation rate $\kappa \int_0^t \abs{\grad \theta^\kappa}_{L^2}^2\rmd s$ becomes \emph{independent} of $\kappa$.
That is, we expect
\begin{equation}\label{an}
\kappa \int_0^t \abs{\grad \theta^\kappa(s)}_{L^2}^2\rmd s \geq \chi>0 \,,
\end{equation}
for some constant $\chi>0$ independent of $\kappa$.
This is behavior known as anomalous dissipation.
The main result in this paper provides an explicit, deterministic example of this.

\subsection{Main Results}
We first produce a divergence free velocity field which exhibits anomalous dissipation for all initial data that is sufficiently close to a non-constant eigenfunction of the Laplacian.

\begin{theorem}[Universal rate near Harmonics]\label{mainthm}
  Fix $T>0$, $d \geq 2$, and $\alpha \in [0,1)$, and let
  \begin{equation*}
    \Psi \defeq \set{ \sin( Mx ) \sin( Ly ), \sin( Mx) \cos( Ly ), \sin( Ly) \cos( Mx ), \cos( Mx ) \cos(Ly) }_{M\geq 0, L \geq 0} - \set{0, 1}\,.
  \end{equation*}
  There exists absolute constants $\epsilon_\alpha, \chi_\alpha >0$, and a divergence-free velocity field
  \begin{equation}\label{e:ureg}
    u\in C^\infty([0,T)\times \mathbb{T}^d) \cap L^1([0,T]; C^\alpha(\mathbb{T}^d)) \cap L^\infty([0,T]\times \mathbb{T}^d) \,,
  \end{equation}
  such that the following holds:
  If $\theta_0 \in H^2(\T^d)$ is mean zero, and there exists $\lambda > 0$ and $\psi \in \Psi$ such that
  \begin{equation*}
    |\theta_0 - \lambda \psi|_{L^2} \le \varepsilon_\alpha |\theta_0|_{L^2}\,,
  \end{equation*}
  then
  \begin{equation}\label{anomaly}
    \kappa \int_0^T |\nabla \theta^\kappa|_{L^2}^2 \rmd t
    \geq \chi_\alpha \abs{\theta_0}^2_{L^2} \,.
  \end{equation}
  \end{theorem}
\begin{remark*}
It  is not difficult to modify the velocity field so that anomalous dissipation occurs for any initial data whose ``width'' of the spectrum is bounded by some finite constant. 
\end{remark*}

For arbitrary $H^2$ initial data, a small modification of the velocity field used above will also exhibit anomalous dissipation.
However the velocity field and the dissipation rate will depend on the data.
\begin{theorem}[Data dependent rate and velocity field]\label{mainthm2}
  Fix $T>0$, $d \geq 2$, $\alpha \in [0,1)$, and a mean-zero $\theta_0\in H^2(\mathbb{T}^d)$. There exists a divergence-free velocity field  \begin{equation*}
    u\in C^\infty([0,T)\times \mathbb{T}^d) \cap L^1([0,T]; C^\alpha(\mathbb{T}^d)) \cap L^\infty([0,T]\times \mathbb{T}^d) \,,
  \end{equation*} and $\chi_\alpha(\theta_0)>0$ 
  so that we have
  \begin{equation}\label{anomaly2}
    \kappa \int_0^T |\nabla \theta^\kappa|_{L^2}^2 \rmd t
    \geq \chi_\alpha(\theta_0) \abs{\theta_0}^2_{L^2}. 
  \end{equation}
\end{theorem}

Our constructions are sharp in the sense that if $\alpha = 1$, then the dissipation must vanish (i.e.\ $\chi_1 = 0$).
In fact, if $u\in L^1([0,T]; W^{1,1}(\mathbb{T}^d))$ then all weak solutions of the inviscid transport equation are renormalized and hence conservative \cite{DiPernaLions89}.
Moreover, since $\theta^\kappa \to \theta$ weakly in $L^2$ and the norms converge (by lower semi-continuity of $L^2$ under weak limits), the convergence is in fact strong and so we must have~$\chi_1 = 0$.
In our construction, the scalar~$\theta^\kappa$ does not retain any H\"older regularity uniformly in $\kappa$ on the whole time interval $[0,T]$. 
As such, our result establishes the sharpness of the Obukhov-Corrsin theory (discussed at the end in \S \ref{disc}) for fields which lose regularity at a single instance in time in the endpoint case of $u\in C^\alpha$ with $\alpha<1$ and $\theta\in C^\beta$ with $\beta=0$.
 In light of this connection, Theorem \ref{mainthm} can be understood also as a proof of the analogue of Onsager's conjecture for passive scalar turbulence in our specific setting. 
 \smallskip

We prove Theorem~\ref{mainthm} by constructing a velocity field which develops smaller and smaller scales with time, mimicking the time development of a turbulence cascade.
As a result, the velocity field has non-trivial energy at  ``infinite frequency'' at the final time, $T$.
At this point in time, the velocity can be made to be H\"older $C^\alpha$ for any $\alpha<1$ but not better.
Due to the precise nature of the construction, we track explicitly the resulting cascade of scalar energy to high-frequency.
The scalar field $\theta^\kappa$ is bounded, but as mentioned above, is not uniformly  H\"older  for any $\beta>0$.

The velocity field we construct alternates horizontal and vertical shears, motivated by the work of Pierrehumbert~\cite{Pierrehumbert94}.
The velocity fields used in~\cite{Pierrehumbert94} involves sinusoidal shears of a single frequency, with a random phase shift.
Our velocity fields, on the other hand, require the use of higher frequencies as time progresses and possess multiple scales.
 \smallskip

We now briefly digress and present an application of Theorem~\ref{mainthm2} showing non-uniqueness of solutions to the transport equation with an irregular drift.
Recall, that solutions to the transport equation with an $L^1([0, T]; W^{1, \infty}(\T^d))$ drift are easily seen to be unique.
Seminal work of DiPerna and Lions~\cites{DiPernaLions89} show that for $L^1(0, T]; W^{1, 1}(\T^d)$ incompressible velocity fields, all weak solutions are renormalized and hence unique.
Ambrosio~\cite{Ambrosio04} extended it further to $L^1( [0, T]; \mathit{BV}(\T^d) )$ incompressible vector fields. 
More generally, uniqueness of weak solutions to the transport equation is closely connected to energy conservation of solutions.
In the DiPerna Lions framework, conservation of energy follows from the so-called re-normalization property.
For lower regularity velocity fields several counterexamples to uniqueness, and consequently to conservation of energy for solutions of the transport equation, are known~\cite{Aizenman78,ColombiniLuoEA03,Depauw03,AlbertiBianchiniEA14,CrippaGusevEA15,ModenaSzekelyhidi18}.
In particular, Alberti et\ al.~\cite{AlbertiBianchiniEA14} abstractly show the existence of a H\"older continuous, time independent, divergence free vector field for which the transport equation does not have a unique solution.
In this direction, we use Theorem~\ref{mainthm2} to produce an explicit, divergence free drift for which the transport equation does not have a unique solution.
In our example the drift is smooth, except at one point in time, and can be chosen to be $L^1_t C^{\alpha}_x$, for any $\alpha < 1$.

\begin{theorem}[Non-uniqueness of the transport equation]\label{nonuniq}
 Fix $T>0$, $d\geq 2$, $\alpha \in [0,1)$ and a mean-zero $\theta_0\in H^2$. Let $u_*$  
  be the divergence-free velocity field from Theorem \ref{mainthm2}.  Let $u$, defined on $[0,2T]$, be
  \begin{equation*}
  u(t) = 
  \begin{cases} \phantom{-}u_*(t) & t\in [0,T),\\
  -u_*(2T-t) & t\in [T,2T].
  \end{cases}
  \end{equation*}
  Then there are at least two weak solutions $\theta, \bar{\theta}\in C_w([0,2T];L^2(\mathbb{T}^d))$ of the transport equation
  \begin{equation}\label{inveqn}
    \partial_t \theta + u \cdot \grad \theta = 0
  \end{equation}
  with initial data $\theta_0$.
\end{theorem}

We prove Theorem~\ref{nonuniq} by constructing one solution as a vanishing viscosity limit, and the other using time reversibility.
The vanishing viscosity solution is dissipative and loses a non-zero fraction of its initial $L^2$-energy.
The time reversible solution, on the other hand, ends with exactly the same $L^2$-energy as it started with.
The full is presented in Section~\ref{s:nonuniq}.
  We conclude this subsection with two remarks concerning anomalous dissipation in the random setting, and magnetic dynamos.

\begin{remark}[Anomalous Dissipation in the Randomized Setting]
  Examples of anomalous in a statistical setting can be found in studies of the Kraichnan model \cite{FalkovichGawedzkiEA01,Gawedzki08}.
  This model advects the scalar by a Gaussian, white-in-time velocity field which is only H\"{o}lder continuous in space and anomalous dissipation for passive scalars can be proved upon taking expectation of \eqref{an} over the random velocity field.
  For precise rigorous statements, see the works  \cite{LeJanRaimond02,LeJanRaimond04}.
  The mechanism for anomalous dissipation discovered in the Kraichnan model and which holds in far greater generality is the breakdown of uniqueness of Lagrangian particle trajectories  or \emph{spontaneous stochasticity} \cite{BernardGawedzkiEA98,DrivasEyink17,EyinkDrivas15}.
  While this phenomenon is expected to be robust in a turbulent setting, the proof of anomalous dissipation and spontaneous stochasticity in the Kraichnan model rely heavily on the Gaussian nature of the advecting velocity and, more importantly, on the white-in-time correlation.
  Moreover, since the velocity field is only distributional in time (formally the temporal regularity is like a derivative of Brownian motion), it is not clear how to generate examples of~\eqref{an} for distributional solutions to the advection diffusion equation in a fixed deterministic velocity field.    
 We remark also that \cite{BedrossianBlumenthalEA19b} studies a related problem of anomalous dissipation of the scalar in forced statistically steady state, allowing for the advecting velocity to be a solution of forced Navier-Stokes with independent forcing random. Namely, \cite{BedrossianBlumenthalEA19b} establishes a constant flux of scalar energy through all small length-scales is established in a permanent regime where scalar energy is input in a (statistically) constant rate.
\end{remark}

 \begin{remark}[Magnetic Dynamo Example]
 Our construction has implications for the existence of a (finite time) magnetic dynamo in two dimensions.  In particular, consider the 2D resistive passive vector equation
 \begin{align*}
 \partial_t B^\kappa+ u \cdot \nabla B^\kappa\ - \ &B^\kappa\cdot \nabla u = \kappa \Delta B^\kappa,\\
 \nabla \cdot B^\kappa =0,  \ \ & \ \ \nabla \cdot u = 0,\\
B^\kappa|_{t=0}&=B_0,
 \end{align*}
 modeling the evolution of a magnetic field $B$ in a prescribed velocity field $u$.
 The unique solution of the above equation can be constructed with a stream function $B^\kappa=\nabla^\perp \psi^\kappa$ solving
 \begin{equation}
 \partial_t \psi^\kappa + u \cdot \nabla \psi^\kappa = \kappa \Delta \psi^\kappa
 \end{equation}
 provided with any  initial data $\psi_0$ with the property that $\nabla^\perp\psi_0=B_0$.  Thus, our results for anomalous dissipation apply to $\psi^\kappa$ which implies that if $u$ is chosen as in Theorem \ref{mainthm} then
 \begin{equation}\label{dynamo}
       \int_0^T |B|_{L^2}^2 \rmd t
      \geq \frac{ \chi }{\kappa}.
 \end{equation}
 This behavior shows unbounded growth of the magnetic field as $\kappa\to 0$ in finite time, seemingly in violation of the 2d anti-dynamo theorems (see chapter 4 of \cite{ArnoldKhesin98}).  However, these results assume advecting velocities are smooth for infinite time at fixed $\kappa$.
 \end{remark}

\subsection{General Criterion for Anomalous Dissipation}
The proof of Theorem~\ref{mainthm} involves comparing $\theta^\kappa$, the solution of the advection diffusion equation~\eqref{visceqn}, to solutions of the transport equation~\eqref{inveqn}.
As a result, we obtain three criterion that guarantee some form of anomalous dissipation.

For each of the results below we fix $T > 0$, assume $u \in L^\infty_\mathrm{loc} ([0, T); W^{1, \infty}(\T^d))$ is divergence free, and let $\theta^\kappa$ and $\theta$ be solutions to~\eqref{visceqn} and~\eqref{inveqn} respectively with the same, $\kappa$ independent, mean zero initial data~$\theta_0 \in H^2(\T^d)$.
The first result is the criterion that will be used in the proof of Theorem~\ref{mainthm}.

\begin{proposition}\label{thm:criterion1}
  If
  \begin{equation}\label{assumProp12}
    \lim_{t\rightarrow T}\int_0^t|\nabla\theta |_{L^2}^2ds=+\infty \,,
    \qquad\text{and}\qquad
    |\theta(t)|_{\dot{H}^1}^2
      \ge c \paren[\big]{
	|\theta(t)|_{\dot{H}^2}+ |\theta^\kappa(t)|_{\dot{H}^2}
	}|\theta_0|_{L^2}
  \end{equation}
  for all $t\in [0,T)$, for a fixed constant $c\in(0,1)$ independent of $\kappa$ and $t$,  then
  \begin{equation*}
  	\begin{split}
  	\kappa \int_0^T |\nb \theta^{\kappa}|_{L^2}^2 dt \ge  \left(\frac{c}{2}\right)^4 |\theta_0|_{L^2}^2.
  	\end{split}
  \end{equation*}
\end{proposition}

Even though Proposition~\ref{thm:criterion1} is what we use in the proof of Theorem~\ref{mainthm}, we note that it involves a condition on \emph{both} $\theta$ and $\theta^\kappa$.
The next two results will involve conditions on the inviscid equation alone.

\begin{proposition}\label{thm:criterion2}
  If
  \begin{equation}\label{e:crit2}
    \int_0^T \abs{\grad \theta(s)}_{L^2}^2 \rmd s = +\infty
    \qquad\text{and}\qquad
    |\theta(t)|_{H^{-1}}
      \leq C\frac{\abs{\theta_0}_{L^2}^2}{\abs{\theta(t)}_{\dot H^1}},
  \end{equation}
 for all $t \in [0, T)$ and some constant $C \geq 1$ independent of $\kappa$ and $t$, then
  \begin{equation*}
 \kappa \int_0^T |\nabla\theta^\kappa|^2_{L^2}\rmd t
      \geq
	  \frac{1}{64 C^2}{\abs{\theta_0}_{L^2}^2}.
  \end{equation*}
\end{proposition}

Recall $\abs{\theta(t)}_{H^{-1}}$ is a measure of the scale to which~$\theta(t)$ is mixed, a notion that will be revisited in the next section (see also~\cites{Thiffeault12} for a review).
Note that interpolation and the fact that $\abs{\theta(t)}_{L^2}$ is conserved guarantees that $\norm{\theta(t)}_{H^{-1}} \geq \abs{\theta_0}_{L^2}^2 / \abs{\theta(t)}_{\dot H^1}$.
Thus the assumption~\eqref{e:crit2} essentially requires $\theta(t)$ to become mixed at a comparable rate.

The proof of Proposition~\ref{thm:criterion2} is short and elementary, and is mainly stated here as it establishes a concrete link between mixing and anomalous dissipation.
It is, however, hard to produce examples of mixing, especially at nearly optimal rates.
In fact, as we will see in the proof, particles advected by velocity field used in Theorems~\ref{mainthm} and~\ref{mainthm2} only travel a finite distance in time~$T$.
Thus, these velocity fields are not even mixing, let alone mixing at the nearly optimal rate required in Proposition~\ref{thm:criterion2}.

Finally, we conclude by stating a criterion for anomalous dissipation that only requires growth of positive norms of~$\theta$, a criterion that is weaker than mixing.
\begin{proposition}\label{thm:criterion3}
  If
  \begin{equation}\label{e:crit3}
    |\theta(t)|_{\dot{H}^{2}}
      \le \frac{C |\theta(t)|_{\dot{H}^1}^2}{\abs{\theta_0}_{L^2}} \,,
    \qquad 
    \frac{1}{C(T-t)}
      \leq \frac{\abs{\theta(t)}_{\dot H^1}}{|\theta_0|_{L^2}} \,,
    \qquad \text{and} \qquad
   |\nabla u(t) |_{L^\infty} \leq  \frac{C}{(T-t)} \,,
  \end{equation}
  for all $t \in [0, T)$, and some constant $C\geq 1$, independent of $\kappa$ and $t$,  then there exists a $\chi>0$ depending only on $C$ 
  \begin{equation*}
   \kappa \int_0^T |\nabla \theta^\kappa|^2_{L^2}\rmd t>\chi|\theta_0|_{L^2}^2.
  \end{equation*}
\end{proposition}
 
Again we note that interpolation forces $\abs{\theta(t)}_{\dot H^2} \geq \abs{\theta(t)}_{\dot H^1}^2 / \abs{\theta_0}_{L^2}$.
The first inequality in~\eqref{e:crit3} assumes that $\abs{\theta(t)}_{\dot H^2}$ does not grow any faster.
Moreover, by Gronwall's lemma we immediately see that last assumption in~\eqref{e:crit3} implies that $\abs{\theta(t)}_{\dot H^1}$ can not grow faster than power of $1 /( T-t)$ as $t \to T$.
The second assumption in~\eqref{e:crit3} requires $\abs{\theta(t)}_{\dot H^1}$ to grow at least linearly in $1/(T-t)$.

\subsection{Connections with Enhanced Dissipation and Mixing}

Enhanced dissipation, anomalous dissipation, and mixing are intrinsically related.
Enhanced dissipation is the notion that solutions to~\eqref{visceqn} dissipate energy faster than $e^{- \kappa t}$, the rate at which solutions to the heat equation (with no advection) dissipate energy.
This occurs when the advection sends some fraction of the total energy to high frequencies.
Using certain assumptions on this rate (specifically~\eqref{e:crit2} and~\eqref{e:crit3}), we showed anomalous dissipation in Propositions~\ref{thm:criterion2} and~\ref{thm:criterion3} respectively.

Mixing, on the other hand, requires \emph{all} energy (in the diffusion free case) to be sent to high frequencies.
When~$\kappa$ is small, one still expects energy to be sent to high frequencies, and so mixing implies enhanced dissipation, at least when $u$ is regular (see for instance~\cite{ConstantinKiselevEA08,CotiZelatiDelgadinoEA18,Wei18,FengIyer19}).
The converse, of course, need not be true:  cellular flows enhance dissipation, but are certainly not mixing~\cite{IyerXuEA19}.

In the context of mixing, Bressan \cite{Bressan03} raised an interesting open problem: \emph{is there a lower bound on the mixing rate of a rough velocity field, in the absence of diffusion?}
More precisely, 
\begin{conj}[Bressan~\cite{Bressan03}]\label{c:Bressan}
If $\theta$ is a solution to~\eqref{inveqn} on the torus, then
\begin{equation}\label{e:mixLower}
  |\theta(t)|_{\mathrm{mix}}\ge
    C_1(\theta_0) \exp\Bigl(
      -C_2(\theta_0)
      \int_0^t |\nabla u(\cdot,s)|_{L^1} \rmd s
    \Bigr)
    \,,
\end{equation}
for some constants $C_1$, $C_2$ that depend on the initial data.
\end{conj}

Here $|\theta|_{\mathrm{mix}}$ is some quantification of the mixing scale of $\theta$.
One common choice is to use multi-scale norms, and set $\abs{\theta}_{\mathrm{mix}} = |\theta-\bar\theta|_{H^{-1}}$ (see~\cite{Thiffeault12} for a comprehensive review).
However, geometric scales, such as those used in~\cite{Bressan03} or~\cite{LunasinLinEA12} may also be used.

A quick application of Gronwall's lemma shows that Conjecture~\ref{c:Bressan} holds if $\abs{\nabla u}_{L^1}$ is replaced by $\abs{\nabla u}_{L^\infty}$.
When $u$ is only  $L^1_t W^{1, p}_t$, solutions to~\eqref{inveqn} need to be interpreted in the renormalized sense~\cite{DiPernaLions89}.
Regularity of these solutions was studied by Crippa and DeLellis~\cite{CrippaDeLellis08,CrippaDeLellis08a}, and their results can be used to prove that~\eqref{e:mixLower} holds if $\abs{\nabla u}_{L^1}$ is replaced by $\abs{\nabla u}_{L^p}$ for any $p > 1$ (see~\cite{IyerKiselevEA14,CrippaDeLellis08a}).
In this case recent results~\cites{AlbertiCrippaEA19,YaoZlatos17,ElgindiZlatos18,BedrossianBlumenthalEA19a} construct explicit examples showing that the lower bound~\eqref{e:mixLower} is indeed attained.
When the velocity field is allowed to be less regular than $L^1_t W^{1,1}_x$,  (for instance if $u \in L^1_t \mathrm{BV}_x$, or even if $u \in L^1_t C^\alpha_x$ with $\alpha < 1$), one can have perfect mixing in finite time.
Indeed, the examples in~\cite{Bressan03,LunasinLinEA12} exhibit situations where $\abs{\theta(t)}_{\mathrm{mix}}$ decreases linearly and hits $0$ in finite time~\cite{AlbertiCrippaEA19}.
However, when $u \in L^1_t W^{1, 1}_x$, as stated in Conjecture~\ref{c:Bressan}, the optimal lower bound on the mixing rate is not known.

In the presence of diffusion, we formulate a version of the above using dissipation enhancement.
First, using Gronwall's lemma and energy methods (see for instance~\cite{Poon96,MilesDoering18}) one can obtain the following double exponential lower bound on the $L^2$ energy\footnote{We remark that it is also unknown whether this double exponential lower bound above is attained for any flow.
In discrete time~\cite{FengIyer19} produce an example where is in fact attained.
In continuous time, however, there are no examples exhibiting the double exponential decay.
Moreover, Miles and Doering~\cite{MilesDoering18} provide numerical evidence and a heuristic argument that the Batchelor scale limits the effectiveness of mixing, suggesting that the $L^2$ energy can only decay exponentially.}:
\begin{equation*}
  |\theta^\kappa|_{L^2}
    \ge
      |\theta_0|_{L^2}
      \exp\Bigg(
	-\frac{\kappa \abs{\grad \theta_0}_{L^2}^2}{\abs{\theta_0}_{L^2}^2}
	\int_0^t 
	\exp\Bigl(
	  C \int_0^s |\nabla u(\cdot, s')|_{L^\infty} \rmd s'
	\Bigr)
	\rmd s\,\Bigg).
\end{equation*}
Here $C$ is an explicit dimensional constant $C$.
If~$u$ is less regular, does $\theta^\kappa$ dissipate at the same rate?
Can it be faster? 
Thus, in the presence of diffusion, we formulate a version of Bressan's conjecture as:
\begin{conj}\label{conj}
If $\theta^\kappa$ is a solution to~\eqref{visceqn} with  $u\in L^1( [0,\infty), W^{1,1}(\mathbb{T}^d))$ and smooth initial data, then there exists a universal rate function $r:=r(\kappa)\to 0$ as $\kappa\to 0$ independent of $u$ such that for all $0\le \kappa\le 1$ 
\begin{equation*}
  |\theta^\kappa|_{L^2}
    \ge
      |\theta_0|_{L^2}
      \exp\Bigg(
	-r(\kappa) C_1(\theta_0)
	\int_0^t 
	\exp\Bigl(
	  C_2 \int_0^s |\nabla u(\cdot, s')|_{L^1} \rmd s'
	\Bigr)
	\rmd s \Bigg).\, \label{quantbnd}
\end{equation*}
Here $C_1 > 0$ is a constant that depends on $\theta_0$, but not~$\kappa$, and $C_2 > 0$ is a universal constant. In particular,
\begin{equation*}
\kappa \int_0^t  |\nabla \theta^\kappa|_{L^2}^2 \rmd s
    \le
      |\theta_0|_{L^2}^2\left(1-
      \exp\Bigl(
	-2r(\kappa) C_1(\theta_0)
	\int_0^t 
	\exp\Bigl(
	  C_2 \int_0^s |\nabla u(\cdot, s')|_{L^1} \rmd s'
	\Bigr)
	\rmd s \Bigr)\,\right).
\end{equation*}
Thus, for $\kappa\ll1$, we have
\be\label{dissbd}
\kappa \int_0^t  |\nabla \theta^\kappa|_{L^2}^2 \rmd s\lesssim r(\kappa).
\ee
\end{conj}
In an earlier version of our paper, we stated Conjecture \ref{conj} with $r(\kappa)=\kappa$. This version of the conjecture was false, as privately communicated to us by Bru\'{e} and Nguyen (private communication).  In work in preparation  \cite{BN}, it is conjectured that \eqref{quantbnd} hold with $r(\kappa)=\ln(\kappa)^{-1}$ and provide some partial results towards this.

We remark that it is also not known whether this conjecture holds with $|\nabla u|_{L^p}$ for any $p \in (1, \infty)$.
Since (morally) enhanced dissipation only requires growth of the $|\theta^\kappa|_{H^1}$ and not actual mixing, this problem appears harder than Bressan's conjecture~\cite{Bressan03}. 
The difficulty is that the $H^1$ norm of the inviscid solution may become infinite immediately~\cite{AlbertiCrippaEA19a} even when $u\in L^\infty_t W^{1,p}$ when $p<\infty$.

Note that main theorem says that one cannot hope to have any such lower bound if we just assume that $u\in L^1_{\mathrm{loc}}([0, \infty); C^\alpha(\T^2))$ if $\alpha<1$.
We further remark that, while the natural place to look for velocity fields breaking this lower bound would be to use rough velocity fields that mix in finite time, it is not easy to rigorously show that mixing implies enhanced dissipation in low regularity settings (see, for example, Theorem 4.4 from \cite{CotiZelatiDelgadinoEA18}).

\subsection{Notation Convention and Plan of this Paper}
For simplicity of presentation, we present the proofs of the main theorem two spatial dimensions, as the generalizing to higher dimension is straightforward.
Without loss of generality, we will also set $T=1$ and subsequently assume that the initial data~$\theta_0$ is always mean zero:
 \[\int_{\mathbb{T}^d} \theta_0(x) \rmd x =0.\]

We use the convention
\begin{equation*}
  |\theta |_{L^2} \defeq \left( \int_{\mathbb{T}^2} |\theta |^2 \rmd x \right)^{{1}/{2}} \,,
  \qquad
  |\theta |_{\dot{H}^1}^2 \defeq
    \sum_{i=1}^2 \abs{\partial_i \theta}_{L^2}^2\,,
  \qquad\text{and}\qquad
  |\theta |_{\dot{H}^2}^2
    \defeq
    \sum_{i,j=1}^2 \abs{\partial_i \partial_j \theta}_{L^2}^2\,,
\end{equation*}
for any function $\theta \colon \T^2 \to \R$.
We will also use
  $|\theta|_{\dot{W}^{1,\infty}}
      \defeq
      \max\{ |\rd_1\theta|_{L^\infty} , |\rd_2\theta|_{L^\infty} \}
  $.
With these conventions, 
\begin{equation*}
  |\theta |_{\dot{H}^1}^2 \le |\theta |_{L^2}|\theta |_{\dot{H}^2} \,,
\end{equation*}
and
\begin{equation*}
  |\theta|_{H^2}^2 = |\theta|_{\dot{H}^2}^2+|\theta|_{\dot{H}^1}^2+|\theta|_{L^2}^2 \,,
  \quad |\theta|_{H^1}^2 =  |\theta|_{\dot{H}^1}^2+|\theta|_{L^2}^2\,,
  \quad |\theta|_{W^{1,\infty}}
    = \max\{ |\theta|_{L^\infty} , |\theta|_{\dot{W}^{1,\infty}} \}
  \,. 
\end{equation*}

To compare quantities that depend on time, we will use $A(t) \approx B(t)$ to mean that $A(t) / B(t)$ is bounded above and below by absolute positive constants.
We will also use $A\ll B$ to mean that $\lim_{t\to 1} A/B=0$.
\medskip

The plan of the paper is as follows.  In \S\ref{sec:inv} we prove the criterion for anomalous dissipation (Propositions~\ref{thm:criterion1}--\ref{thm:criterion3}).
In \S\ref{sec:ex}, we construct the velocity field used in Theorem~\ref{mainthm}, and prove Theorems~\ref{mainthm}--\ref{mainthm2}.
In \S\ref{s:nonuniq}, we use Theorem~\ref{mainthm2} to prove non-uniqueness of weak solutions to the transport equation.
Finally, in \S\ref{disc}, we discuss the connection of our construction to the sharpness of the Obukhov-Corrsin scaling theory of passive scalar turbulence and conclude with an open question.

\section{Criteria for Anomalous Dissipation}\label{sec:inv}

In this section we prove Propositions~\ref{thm:criterion1}--\ref{thm:criterion3}.
The first result will also be used in the proof of Theorem~\ref{mainthm}.

\subsection{Inviscid Growth Criterion with an Assumption on \texorpdfstring{$|\theta^\kappa|_{H^2}$}{the H2 Norm}}

\begin{proof}[Proof of Proposition \ref{thm:criterion1}]
  For simplicity, and without loss of generality, we assume $|\theta_0|_{L^2}$ = 1.
  Since
\[\frac{1}{2} \partial_t |\theta -\theta^\kappa|_{L^2}^2=\kappa\int\Delta \theta^\kappa (\theta^\kappa-\theta)\rmd x= -\kappa | \theta^\kappa|_{\dot{H}^1}^2 + \kappa\int\nabla \theta^\kappa\cdot \nabla \theta\rmd x.\]
Thus, upon integration and using the Cauchy-Schwarz inequality we have
\begin{equation*}
  \frac{1}{2} |\theta -\theta^\kappa|_{L^2}^2(t)
    \le \paren[\Big]{\kappa\int_0^1|\theta^\kappa|_{\dot{H}^1}^2\rmd s}^{1/2}
      \paren[\Big]{\kappa\int_0^t|\theta |_{\dot{H}^1}^2\rmd s}^{1/2} \,,
\end{equation*}
for all $t<1$. 
Assume toward a contradiction that there exists a sequence $\kappa_k\rightarrow 0$ so that 
\begin{equation}\label{contdelt}
\delta_k \defeq \kappa_k\int_0^1|\theta ^{\kappa_k}|_{\dot{H}^1}^2\rmd t  < \chi   \qquad \text{as}\qquad k\to \infty
\end{equation}
where $\chi \defeq (c/2)^4.$ 
Let $T_k<1$ be such that 
\begin{equation}\label{Tkdef}
\kappa_k\int_0^{T_k}|\theta |_{\dot{H}^1}^2\rmd t=1.
\end{equation}
Note that $T_k\to 1$ as $k\to\infty$.
We have that \[\sup_{t\le T_k}|\theta -\theta^{\kappa_k}|_{L^2}\le \sqrt{2}{\delta_k}^{1/4}\le \sqrt{2}{\chi}^{1/4}.\]
On the other hand, we have by interpolation and our hypothesis \eqref{assumProp12} that for $t\le T_k$, 
\[|\theta -\theta^\kappa|_{\dot{H}^1}^2\le |\theta -\theta^\kappa|_{L^2}|\theta -\theta^\kappa|_{\dot{H}^2}\le   \frac{\sqrt{2}{\chi}^{1/4}}{c}|\theta |_{\dot{H}^1}^2\]
and so  $|\theta -\theta^\kappa|_{\dot{H}^1}\leq\frac{2^{1/4}{\chi}^{1/8}}{\sqrt{c}}|\theta |_{\dot{H}^1}$. By the reverse triangle inequality we have
\begin{equation}
|\theta^\kappa|_{\dot{H}^1} \geq \Big||\theta|_{\dot{H}^1}- |\theta^\kappa- \theta|_{\dot{H}^1}\Big| \geq \left(1-\frac{2^{1/4}{\chi}^{1/8}}{\sqrt{c}}\right)|\theta |_{\dot{H}^1}
\end{equation}
Thus we have
\begin{equation}\label{conteqn}
\kappa_k\int_0^{T_k}|\theta^\kappa|_{\dot{H}^1}^2\rmd t\ge  \kappa_k\left(1-\frac{2^{1/4}{\chi}^{1/8}}{\sqrt{c}}\right)^2 \int_0^{T_k}|\theta|_{\dot{H}^1}^2\rmd t =\left(1-\frac{2^{1/4}{\chi}^{1/8}}{\sqrt{c}}\right)^2.
\end{equation}
Thus, so long as $\chi$ satisfies the following inequality
\begin{equation}\label{chiinq}
\chi \leq \left(1-\frac{2^{1/4}{\chi}^{1/8}}{\sqrt{c}}\right)^2
\end{equation}
the right-hand-side of \eqref{conteqn} exceeds $\chi$ contradicting \eqref{contdelt}.   Since $\chi=(c/2)^4$ we see that \eqref{chiinq} becomes
$
\left(\frac{c}{4}\right)^2 \leq 1-\left(\frac{1}{4}\right)^{1/4}.
$
Since $c<1$ and $1/16< 1-\left({1}/{2}\right)^{1/4}$, \eqref{chiinq} is satisfied thereby concluding the proof.
\end{proof}

\subsection{Inviscid Mixing Criterion}

\begin{proof}[Proof of Proposition \ref{thm:criterion2}]
  For simplicity, and without loss of generality, we again assume $\abs{\theta_0}_{L^2} = 1$. 
  Following the proof of Proposition \ref{thm:criterion1}, we again assume \eqref{contdelt} but now with $\chi \defeq   \frac{1}{2^6 C^2}$. 
  Defining $T_k$ as in \eqref{Tkdef}, we find again
   \[\sup_{t\le T_k}|\theta -\theta^{\kappa_k}|_{L^2}\le \sqrt{2}{\chi}^{1/4}=  \frac{1}{2}{C}^{-1/2}\leq 1/2\]
  since  the constant in~\eqref{e:crit2} satisfies $C\geq 1$.
  Now given $N\in\mathbb{N}$, let $\mathbf{P}_{>N}$ be the projection onto Fourier frequencies higher than $N$.
  Recalling  $|\theta|_{L^2}=1$, we see
  \begin{align*}
    |\mathbf{P}_{>N}\theta^\kappa|_{L^2}
      &\ge |\mathbf{P}_{>N}\theta|_{L^2}-|\mathbf{P}_{>N}(\theta-\theta^\kappa)|_{L^2}
      \ge
	|\mathbf{P}_{>N}\theta|_{L^2}
	- \frac{1}{2}
    \\
      &\ge 1-N^2|\theta|_{H^{-1}}^2-\frac{1}{2}
      \geq \frac{1}{2} 
	-\frac{C^2N^2}{\abs{\theta(t)}_{\dot H^1}^2} \,,
  \end{align*}
where $C$ is the constant in~\eqref{e:crit2}.
Letting $\lambda(t) \defeq \abs{\theta(t)}_{\dot H^1}^2$ and $N^2\defeq {\lambda(t)}/{(2C)^2}$ we see
\begin{equation*}
  |\mathbf{P}_{>N}\theta^\kappa|_{L^2} \ge\frac{1}{4}  \qquad \text{which implies} \qquad |\nabla\theta^\kappa|^2_{L^2}\ge \frac{\lambda(t)}{16C^2}.
\end{equation*}
Consequently, \[
\kappa\int_0^{T_\kappa}|\nabla\theta^\kappa|_{L^2}^2\rmd t\ge \frac{1}{16C^2}=4\chi\,,\] 
contradicting the assumption \eqref{contdelt} and concluding the proof.
  \end{proof}
  
\subsection{Inviscid Growth Criterion with Bounds on \texorpdfstring{$\nabla u$}{Gradient u}}  

\begin{proof}[Proof of Proposition \ref{thm:criterion3}]
We assume $|\theta_0|_{L^2}^2=1$ and $T=1$ and define $\lambda(t)=1/(1-t)$.  Define 
\begin{equation}
  \chi^\kappa(t,\kappa) \defeq \kappa \int_0^{t} |\theta^\kappa(t)|_{\dot{H}^1}^2\rmd t, \qquad \chi^0(t,\kappa) \defeq \kappa \int_0^{t} |\theta(t)|_{\dot{H}^1}^2\rmd t.
\end{equation}
For the sake of contradiction, suppose that
\begin{equation} \label{contass}
\chi^\kappa(1,\kappa)\to 0\qquad\text{as} \qquad  \kappa\to 0.
\end{equation}
Note that we have the following lower bound on the dissipation for all $t\in [0,1]$
\begin{equation}\label{chibndbl}
\chi^\kappa(1,\kappa)\ge \chi^\kappa(t,\kappa) \ge \frac{1}{2}\chi^0(t,\kappa) -\kappa \int_0^t|\theta^\kappa(s)-\theta(s)|_{\dot{H}^1}^2\rmd s.
\end{equation}
For any sequence $\kappa_i\to 0$ we have
\begin{equation}
 \chi^0(1-\kappa_i,\kappa_i)\approx\kappa_i \int_0^{1-\kappa_i}\lambda(s)\rmd s=\frac{1}{2} .
\end{equation}
Now fix $i$ and denote $\kappa=\kappa_i$. We consider now two separate cases
\begin{enumerate}
\item there exists $\lambda \in (0.5, 1.5)$ and $\mu \in (2,3)$ such that 
\begin{equation}
\int_{1-\mu \kappa}^{1-\lambda \kappa} |\theta |_{H^2}^2 \rmd s > \chi^\kappa(1,\kappa)^{1/100} \int_{1-\mu \kappa}^{1-\lambda \kappa} |\theta^\kappa|_{H^2}^2 \rmd s 
\end{equation}
\item for all  $\lambda \in (0.5, 1.5)$ and $\mu \in (2,3)$ we have
\begin{equation}
\int_{1-\mu \kappa}^{1-\lambda \kappa} |\theta |_{H^2}^2 \rmd s \le  \chi^\kappa(1,\kappa)^{1/100} \int_{1-\mu \kappa}^{1-\lambda \kappa} |\theta^\kappa|_{H^2}^2 \rmd s 
\end{equation}
\end{enumerate}
\noindent\textbf{Case $(1)$:}  we use the equation for the difference,
\begin{equation}\label{diffeq}
\partial_t (\theta^\kappa - \theta)+u \cdot \nabla (\theta^\kappa - \theta) = \kappa \Delta \theta^\kappa, \qquad (\theta^\kappa - \theta)|_{t=0}=0
\end{equation}
to find
\begin{equation}
|\theta^\kappa(t) - \theta(t)|_{L^2} \le 2\sqrt{\chi^\kappa(t,\kappa) \chi^0(t,\kappa)}.
\end{equation}
Then, we have by interpolation
\begin{align*}
|\theta^\kappa(t) - \theta(t)|_{H^1}^2&\le |\theta^\kappa(t) - \theta(t)|_{L^2} |\theta^\kappa(t) - \theta(t)|_{H^2} \\
&\le2  \sqrt{\chi^\kappa(t,\kappa) \chi^0(t,\kappa)}|\theta^\kappa(t) - \theta(t)|_{H^2}\\
&\le 2 \sqrt{\chi^\kappa(t,\kappa) \chi^0(t,\kappa)} \left(|\theta(t)|_{H^2}+ |\theta^\kappa(t)|_{H^2}\right).
\end{align*}
Recalling that $ \chi^0(1-\lambda\kappa,\kappa)$ bounded independent of $\kappa$, for an appropriate choice of $\lambda,\mu$ we have
\begin{align*}
\int_{1-\mu \kappa}^{1-\lambda \kappa}  |\theta^\kappa(t) - \theta(t)|_{H^1}^2\rmd t &\le2\sqrt{\chi^\kappa(1,\kappa) \chi^0(1-\lambda\kappa,\kappa)} \int_{1-\mu \kappa}^{1-\lambda \kappa}   \left(|\theta(t)|_{H^2}+ |\theta^\kappa(t)|_{H^2}\right)\rmd t\\
&\lesssim\sqrt{\chi^\kappa(1,\kappa) } \kappa^{1/2}  \left(\int_{1-\mu \kappa}^{1-\lambda \kappa}   \left(|\theta(t)|_{H^2}^2+ |\theta^\kappa(t)|_{H^2}^2\right)\rmd t\right)^{1/2}\\
&\lesssim \chi^\kappa(1,\kappa)^{1/2-1/50} \kappa^{1/2} \left(\int_{1-\mu \kappa}^{1-\lambda \kappa}   |\theta(t)|_{H^2}^2\rmd t\right)^{1/2}\\
&\lesssim  \chi^\kappa(1,\kappa)^{1/2-1/50}\kappa^{1/2} \left(\int_{1-\mu \kappa}^{1-\lambda \kappa}   |\theta(t)|_{H^1}^4\rmd t\right)^{1/2}\\
&\approx \chi^\kappa(1,\kappa)^{1/2-1/50} \int_{1-\mu \kappa}^{1-\lambda \kappa}   |\theta(t)|_{H^1}^2\rmd t\\
&\le \frac{1}{4} \int_{1-\mu \kappa}^{1-\lambda \kappa}   |\theta(t)|_{H^1}^2\rmd t \qquad \text{for sufficiently small $\kappa>0$},
\end{align*}
where we have used the properties of the inviscid solution in the second to last two lines.  In particular, $|\theta|_{H^1}\approx \kappa^{-1}$ on that time interval. Thus, we have from \eqref{chibndbl}
\begin{equation}\label{contbd}
\chi^\kappa(1,\kappa)\ge \kappa \int_{1-\mu \kappa}^{1-\lambda \kappa} |\theta^\kappa(s)|_{H^1}^2\rmd s   \ge\frac{1}{2} \kappa \int_{1-\mu \kappa}^{1-\lambda \kappa} |\theta (s)|_{H^1}^2\rmd s\ge c>0
\end{equation}
which gives a contradiction with  \eqref{contass}.
\bigskip

\noindent\textbf{Case $(2)$:} From \eqref{diffeq} we obtain the $\dot{H}^1$ balance
\begin{equation}
\frac{1}{2} \partial_t |\theta^\kappa - \theta|_{\dot{H}^1}^2 \le |\nabla u|_{L^\infty}  |\theta^\kappa - \theta|_{\dot{H}^1}^2 + \frac{\kappa}{2}\left(|\theta |_{\dot{H}^2}^2 - |\theta^\kappa|_{\dot{H}^2}^2\right).
\end{equation}
Integrating the above from $s$ to $t$ and denoting $K(t,s)\defeq \exp\left(2\int_s^t  |\nabla u(r)|_{L^\infty} \rmd r\right)$ we find
\begin{align}\nonumber
|\theta^\kappa(t) - \theta(t)|_{\dot{H}^1}^2 &\le K(t,s) |\theta^\kappa(s) - \theta(s)|_{\dot{H}^1}^2 + \kappa \int_s^t K(t,s') \left(|\theta (s')|_{\dot{H}^2}^2 - |\theta^\kappa(s')|_{\dot{H}^2}^2\right)\rmd s'\\ \label{H1bal}
 &\le K(t,s) |\theta^\kappa(s) - \theta(s)|_{\dot{H}^1}^2 + \kappa K(t,s) \int_s^t |\theta (s')|_{\dot{H}^2}^2\rmd s' - \kappa\int_s^t |\theta^\kappa(s')|_{\dot{H}^2}^2\rmd s'.
\end{align}
We restrict the above to $s\in (1-3\kappa, 1-2\kappa)$ to any $t\in (1-1.5 \kappa, 1- 0.5\kappa)$. 
First we remark that the integrating factor is uniformly bounded on this interval, namely
 \begin{align*}
1\leq K(t,s)&\le \exp\left(2\int_s^t  |\nabla u(r)|_{L^\infty} \rmd r\right) \le  \exp\left(2\int_{1-3\kappa}^{1-0.5 \kappa}(1-r)^{-1} \rmd r\right)  \le \Gamma.
 \end{align*}
 Under the working hypothesis and the assumption \eqref{contass}
\begin{align}\nonumber
|\theta^\kappa(t) - \theta(t)|_{\dot{H}^1}^2 
 &\leq \Gamma |\theta^\kappa(s) - \theta(s)|_{\dot{H}^1}^2 + \kappa  \left(\Gamma  - \chi^\kappa(1,\kappa)^{-1/100} \right)\int_s^t |\theta (s')|_{\dot{H}^2}^2\rmd s'\\
  &\le \Gamma\left( |\theta^\kappa(s)|_{\dot{H}^1}^2+  |\theta (s)|_{H^1}^2\right) -M \kappa \int_s^t |\theta (s')|_{\dot{H}^2}^2\rmd s', \qquad \text{for any}  \qquad M>0 \nonumber
\end{align}
for some sufficiently small $\kappa$.  Now, under the assumption on the inviscid solution, we have
\begin{equation}
|\theta (s)|_{\dot{H}^1}^2\approx (1-s)^{-2}\approx \kappa^{-2},
\end{equation}
\begin{equation}
 \kappa \int_s^t |\theta (s')|_{\dot{H}^2}^2\rmd s'\approx \kappa \int_s^t |\theta (s')|_{\dot{H}^1}^4\rmd s' \approx  \kappa^{-2} ,
\end{equation}
so that we obtain
\begin{align}\nonumber
|\theta^\kappa(t) - \theta(t)|_{\dot{H}^1}^2 
  &\le  \Gamma |\theta^\kappa(s)|_{\dot{H}^1}^2+   \kappa^{-2}(C_1\Gamma -C_2M) \le \Gamma |\theta^\kappa(s)|_{\dot{H}^1}^2 
\end{align}
for two appropriate constants $C_1,C_2>0$. Choosing $\kappa$ sufficiently small so that $M>C_1\Gamma/C_2$ we find upon
integrating the above inequality $s$ from $1-3\kappa$ to $1-2\kappa$ and in $t$ from $1-2\kappa$ to $1-\kappa$ that
\begin{equation}
\kappa\int_{1-2\kappa}^{1-\kappa} |\theta^\kappa(t) - \theta(t)|_{\dot{H}^1}^2\rmd t \le  \Gamma \kappa \int_{0}^{1-\kappa}|\theta^\kappa(s)|_{\dot{H}^1}^2\rmd s.
\end{equation}
Since the right-hand-side above vanishes by assumption, we can choose $\kappa$ sufficiently small such that
\begin{equation}
\kappa \int_{1-2\kappa}^{1-\kappa} |\theta^\kappa(t) - \theta(t)|_{\dot{H}^1}^2 \rmd t\le \frac{1}{4} \chi^0(1-\kappa,\kappa)
\end{equation}
as desired to produce a contradiction by means of \eqref{contbd} as before.
\end{proof}

\section{Construction of the Example}\label{sec:ex}


In this section, we establish Theorems \ref{mainthm} and \ref{mainthm2} by providing an example of velocity field satisfying Proposition \ref{thm:criterion1}.
As mentioned earlier, we will assume for simplicity that $T = 1$ and $d = 2$ in the statements of the theorems.
Moreover, contrary to the regularity stated in~\eqref{e:ureg}, we construct a velocity in the class
\begin{equation*}
  u\in  L^\infty_{\mathrm{loc}}([0,1); W^{2,\infty}(\mathbb{T}^2))\cap L^1([0,1]; C^\alpha(\mathbb{T}^2)) \cap L^\infty([0,1]\times \mathbb{T}^2).
\end{equation*}
 
Our construction is based on the following smoothed-out `sawtooth' function. Given $\frac{\pi}{2} > \epsilon>0$, we define $S_\epsilon:\mathbb{T} = (\mathbb{R}/2\pi\mathbb{Z})\rightarrow \mathbb{R}$ to be odd with respect to $0$, even with respect to {$\frac{\pi}{2}$}, and \[ S_\epsilon(x)=\begin{cases} 
      x & 0\le x\le {\frac{\pi}{2}} -\epsilon \\
      x -\frac{1}{2\epsilon}(x-\frac{\pi}{2}+\epsilon)^2  & \frac{\pi}{2}-\epsilon< x\le \frac{\pi}{2} 
   \end{cases}
\]
Observe that $S_\epsilon\in W^{2,\infty}(\mathbb{T})$ and \[|S_\epsilon'|_{L^\infty}\le 1\qquad |S_\epsilon''|_{L^\infty}\le \frac{1}{\epsilon}.\]
Let us first fix a sequence to time steps  $\{t_j\}_{j\in\mathbb{N}}$, a sequence of regularizations $\eps_j$, and a sequence of frequencies $\{N_j\}_{j\in\mathbb{N}}$. In practice, $t_j$ and $\eps_j$ are going to be chosen to be decreasing and summable while $N_j$ will be chosen to be rapidly increasing. Next, define measure preserving transformations $\{\mathcal{T}_j\}_{j\in\mathbb{N}}$ by \[\mathcal{T}_j(x,y)=\begin{cases} (x+t_jS_{\eps_j}(N_j y),y) & j\,\, \text{odd}\\  (x,y+t_jS_{\eps_j}(N_j x)) & j\,\, \text{even}\end{cases}.\]
Now define \[\mathcal{U}_j\defeq\mathcal{T}_1\circ \mathcal{T}_2\circ \mathcal{T}_3\circ \dots \mathcal{T}_j.\] Set $T_j = \sum_{i=0}^j t_i$ with $t_0 = 0$. 
Note that   $\theta_0 \circ \mathcal{U}_j = \theta(T_j)$ where $\theta(t)$ is the solution of \begin{equation*}
\begin{split}
\partial_t \theta + u\cdot\nabla \theta = 0, \qquad \theta(t=0) = \theta_0, 
\end{split}
\end{equation*} with $u(t)$ for $t \in [T_i,T_{i+1})$ given by \begin{equation}\label{eq:velocity}
\begin{split}
u(t, x,y) = \begin{cases}
\begin{pmatrix}
S_{\eps_i}(N_iy) \\
0
\end{pmatrix}  & i \mbox{ even},  \\
\begin{pmatrix}
0 \\
S_{\eps_i}(N_ix) 
\end{pmatrix} & i \mbox{ odd. }
\end{cases} 
\end{split}
\end{equation} In the following sections, we proceed to check the conditions in Proposition \ref{thm:criterion1}. To treat the case of $\alpha>0$ small, we modify $u(t)$ to be trivial for $t \le T_{j}$ with $j<\frac{1}{\alpha}$ for a technical reason; see Lemma \ref{thm:inviscid-bounds}. 

\begin{figure}[htb]\centering
  \begin{subfigure}[b]{0.47\linewidth}
    \includegraphics[width=.9\columnwidth,  height=.8\columnwidth]{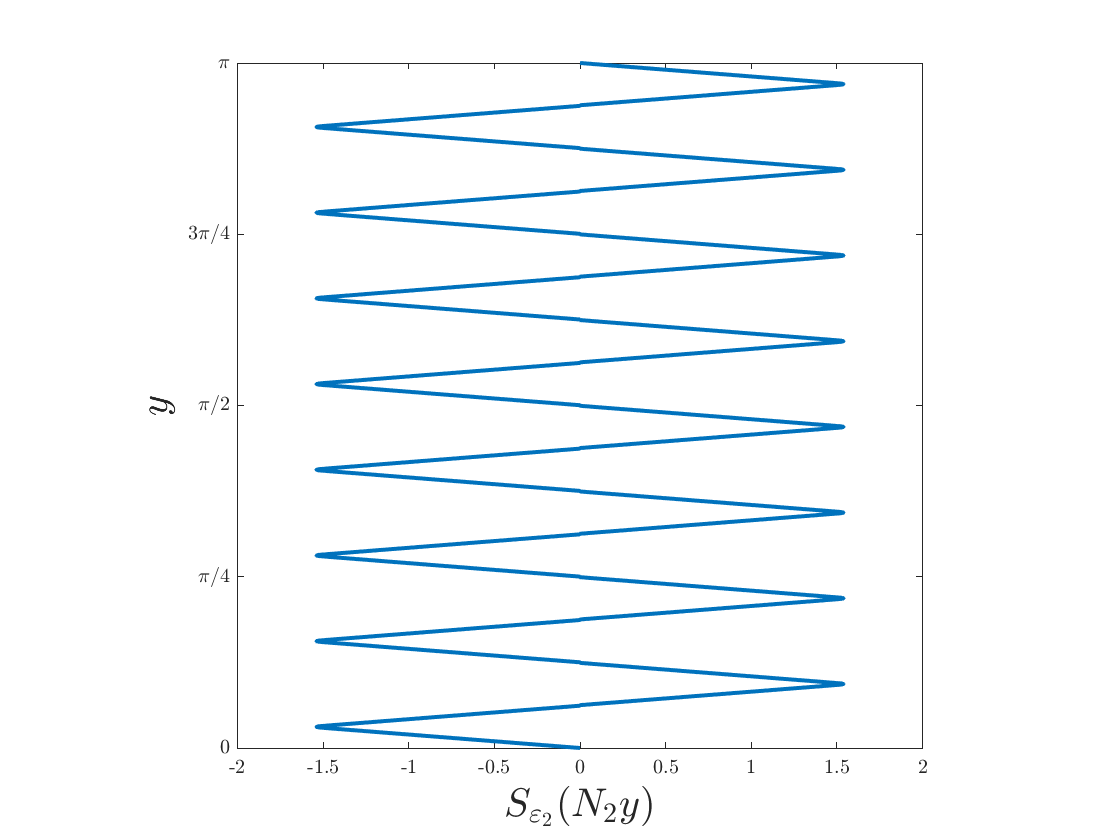} 
  \end{subfigure}
  \begin{subfigure}[b]{0.47\linewidth}
    \includegraphics[width=.9\columnwidth,  height=.8\columnwidth]{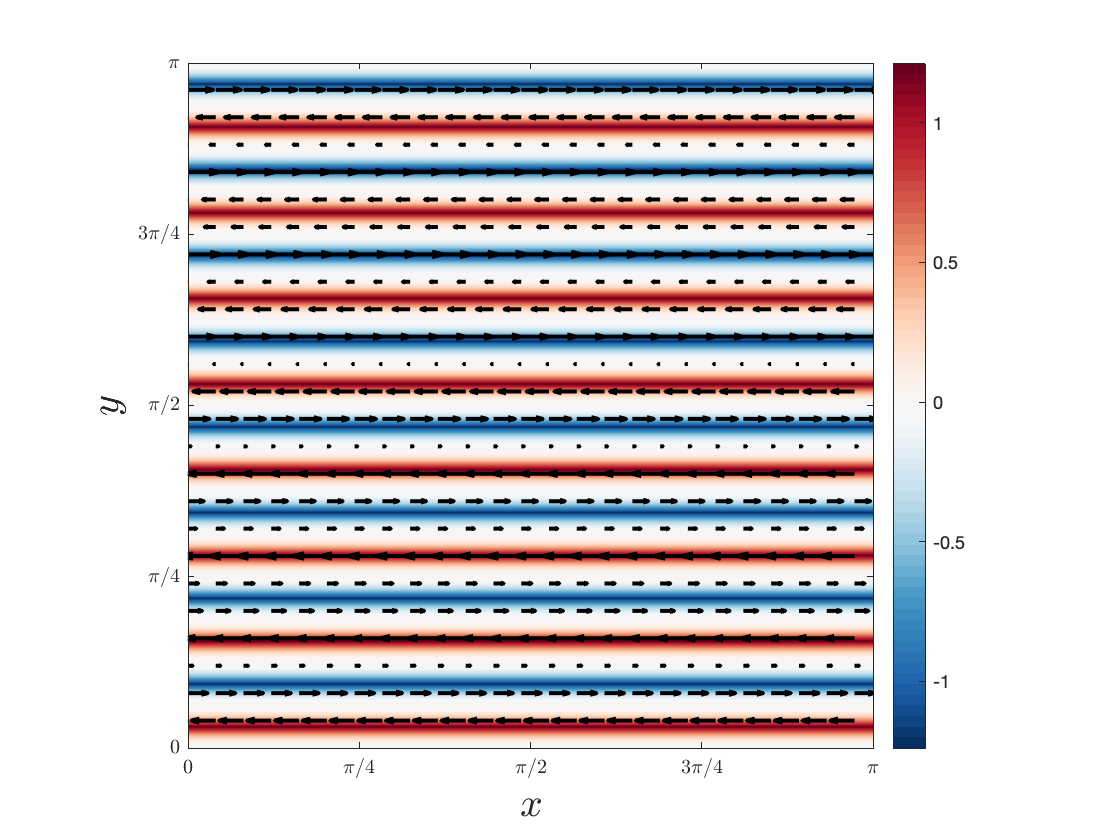} 
  \end{subfigure}
  \\
  \begin{subfigure}[b]{0.47\linewidth}
    \includegraphics[width=.9\columnwidth,  height=.8\columnwidth]{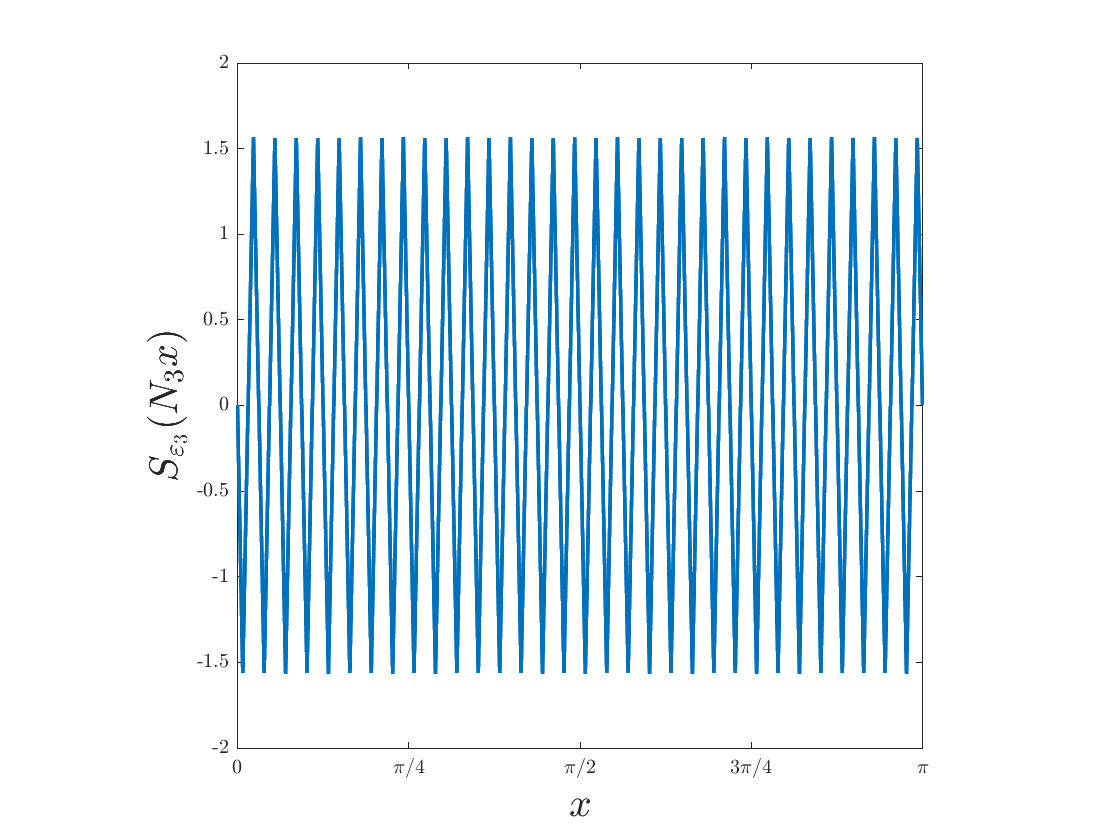} 
  \end{subfigure}
  \begin{subfigure}[b]{0.47\linewidth}
    \includegraphics[width=.9\columnwidth,  height=.8\columnwidth]{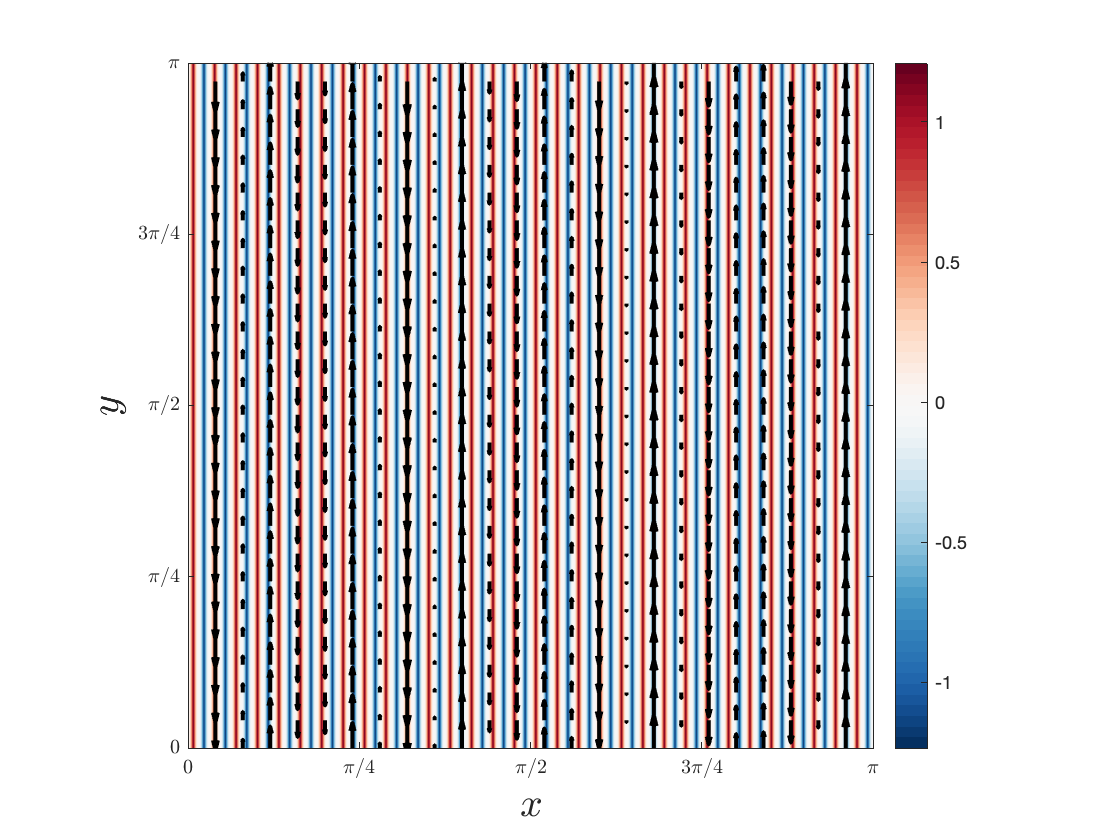} 
  \end{subfigure}\\  
  \caption{Fix $\alpha=1$ and $\epsilon_i$ and $N_i$ as above.
    The two left panels depict shear profiles, and two right panels represent contour plots of the corresponding stream function with velocity vectors superimposed.
    The two top panels correspond to $t_2= 2^{-2}$, and the two Bottom panels correspond to $t_3= 2^{-3}$.
  }
\label{fig1}
\end{figure}

\begin{remark} In the above, we have constructed $u \in L^\infty_\mathrm{loc}( [0, T); W^{1, \infty}(\T^2))$ rather than $u \in C^\infty( [0, 1) \times \T^2)$.
	The modification required to obtain smooth velocities is straightforward; it is accomplish by modifying the shear profile, $S_\ve$  to be $C^\infty$, instead of~$C^2$,  and adding amplitude functions which smoothly turn on and off the shears over each time interval in the construction.
	The window over which the shears are turned on and off are taken small to start with, and can be taken to decrease as time progresses and $N$ increases in the construction.
\end{remark}

\subsection{Inviscid Bounds}

In this section, we prove the following
\begin{lemma}\label{thm:inviscid-bounds} 
For any $\alpha >0$, let \[t_j=2^{-j},\qquad N_j=2^{{(1+\alpha)}j},\qquad \eps_j={\exp\left(-30(1 + \frac{1}{2^\alpha-1})\right)}\cdot 2^{-2j}.\] {Assume that $\theta_0$ is given by one of the following trigonometric functions: \begin{equation*}
	\begin{split}
	\sin(Mx)\sin(Ly),\quad \sin(Mx)\cos(Ly),\quad\cos(Lx)\sin(My),\quad \cos(Lx)\cos(My)
	\end{split}
	\end{equation*} for some integers $M\ge1$ and $L\ge 0$. Define $\theta_j(x,y)=\theta_0\circ \mathcal{U}_j$ where}\begin{equation*}
\begin{split}
\mathcal{U}_j \defeq \mathrm{Id}, \quad j \le \frac{1}{\alpha} 
\end{split}
\end{equation*} and \begin{equation*}
\begin{split}
\mathcal{U}_j \defeq \mathcal{T}_{J(\alpha)} \circ \cdots \circ \mathcal{T}_j, \quad j \ge J(\alpha)\defeq \left\lfloor \frac{1}{\alpha} \right\rfloor + 1 . 
\end{split}
\end{equation*} Here $\left\lfloor \frac{1}{\alpha} \right \rfloor$ denotes the largest integer not exceeding $\frac{1}{\alpha}$.  Then, $\theta_j$ satisfy:
\[|\theta _j|_{H^1}\ge {c_\alpha} 2^{\frac{{\alpha}j(j+1)}{2}}{|\theta_0|_{H^1}} ,\qquad |\theta _j|_{W^{1,\infty}}\le {C_\alpha} |\theta _j|_{H^1}, \qquad |\theta _j|_{H^2}\le {C_\alpha}|\theta _j|_{H^1}^2,\] for some constants $C_\alpha,c_\alpha>0$ independent of $j$. \end{lemma}

\begin{remark}
It is easy to see that the choice of $t_j, N_j, \eps_j$ in the above gives us $|\nabla u (t)|_{L^\infty} \approx \frac{1}{(1-t)^2}$ while $|\theta (t)|_{H^1}^2\approx |\theta (t)|_{H^{{2}}}\gg \frac{1}{(1-t)^2}$ where $u$ is defined as in \eqref{eq:velocity}. 
\end{remark}

\begin{proof} We shall assume for simplicity that $\theta_0 = \sin(Mx)\sin(Ly)$ with $M \ge L$. The proof for other trigonometric functions are almost identical, as long as $M\ge L$. We shall sketch necessary modifications to the proof in the case $L >M$ at the end.
{We now observe that \begin{equation*}
	\begin{split}
	|\theta_0|_{\dot{W}^{1,\infty}} \le 4|\theta_0|_{\dot{H}^1}, \quad |\theta_0|_{\dot{H}^2} \le 4 |\theta_0|_{\dot{H}^1}^2.
	\end{split}
	\end{equation*} Fix} $j\in \mathbb{N}\cup\{0\}$ and let $i_{j}=j-1\,\, \text{mod}\,\, 2$ and $i_{j+1}=j \,\, \text{mod}\,\, 2.$ Assuming for a moment that $j$ is odd,  we compute: 
\begin{align*}
\rd_{2} \theta_{j+1}(x,y) &= \rd_{2}\theta_j(x + t_j S_{\eps_j}(N_jy),y ) + t_jN_j S'_{\eps_j}(N_jy) \rd_{1}\theta_j(x + t_j S_{\eps_j}(N_jy),y ) , \\
\rd_{1} \theta_{j+1}(x,y) &= \rd_{1}\theta_j(x + t_j S_{\eps_j}(N_jy),y ). 
\end{align*}
Moreover, 
\begin{align*}
\rd_{11}\theta_{j+1}(x,y) &= \rd_{11} \theta_j(x + t_j S_{\eps_j}(N_jy),y ), \\
\rd_{12} \theta_{j+1}(x,y) &= \rd_{12}\theta_j(x + t_j S_{\eps_j}(N_jy),y ) + t_jN_j S'_{\eps_j}(N_jy) \rd_{11}\theta_j(x + t_j S_{\eps_j}(N_jy),y ),\\
\rd_{22} \theta_{j+1}(x,y) &= \rd_{22}\theta_j(x + t_j S_{\eps_j}(N_jy),y ) + t_jN_j^2 S_{\eps_j}''(N_j y) \rd_1\theta_j(x + t_j S_{\eps_j}(N_jy),y ) \\
& \quad + (t_jN_j)^2 (S'_{\eps_j}(N_jy))^2 \rd_{11} \theta_j(x + t_j S_{\eps_j}(N_jy),y ). 
\end{align*} 
We have similar formulas when $j$ is odd, exchanging the roles of $x$ and $y$. Now the upper bounds on $\theta_{j+1}$ in $\dot{W}^{1,\infty}$ and $\dot{H}^1$ are easy to get. 
\medskip

\emph{Upper Bounds:}
From the Lipschitz property of the profile $S_\epsilon$, it is easy to see that
\begin{align*}
|\theta _{j+1}|_{\dot{W}^{1,\infty}} &\le t_j N_j |\theta _j|_{\dot{W}^{1,\infty}}+|\theta _j|_{\dot{W}^{1,\infty}}\\
&\le t_j N_j |\theta _j|_{\dot{W}^{1,\infty}}\Big(1+\frac{1}{t_j N_j}\Big).
\end{align*}
Similarly, we have
 \begin{equation*}
\begin{split}
|\theta _{j+1}|_{\dot{H}^{1}}\le t_j N_j |\theta _j|_{\dot{H}^{1}}\Big(1+\frac{1}{t_j N_j}\Big). 
\end{split}
\end{equation*} 
Recalling that  $t_j=2^{-j}$ and $ N_j=2^{(1+\alpha)j}$, 
notice  \[\prod_{j=1}^{\infty} \Big(1+\frac{1}{t_j N_j}\Big)=\prod_{j=1}^\infty \Big(1+2^{-{\alpha}j}\Big)\le {\exp\left(\frac{1}{2^\alpha-1}\right)} \] by taking the $\log$ of the infinite product and using the fact that $\log(1+x)\le x$ for $x\ge 0$. Then we have 
\[|\theta _{j+1}|_{\dot{W}^{1,\infty}}\le  {\exp\left(\frac{1}{2^\alpha-1}\right)} \cdot 2^{\frac{{\alpha}j(j+1)}{2}} {|\theta_0|_{\dot{W}^{1,\infty}}}. \] The same upper bound holds for $|\theta_{j+1}|_{\dot{H}^1}$: \begin{equation*}
\begin{split}
|\theta _{j+1}|_{\dot{H}^{1}}\le  {\exp\left(\frac{1}{2^\alpha-1}\right)} \cdot 2^{\frac{{\alpha}j(j+1)}{2}} |\theta_0|_{\dot{H}^{1}}. 
\end{split}
\end{equation*}
\medskip

\emph{Lower Bounds:}
Note that $|S_{\eps_j}'(z)| = 1$ except for the region $|z-\frac{\pi}{2}|<\eps_j$. We bound the contribution from this region using the $\dot{W}^{1,\infty}$ norm:
\begin{align*}
|\partial_{i_{j+1}} \theta_{j+1}|_{L^2}&\ge t_j N_j |\partial_{i_j}\theta_j|_{L^2}-\sqrt{\eps_j} t_j N_j|\theta _j|_{\dot{W}^{1,\infty}} - |\partial_{i_{j+1}} \theta_j|_{L^2}\\
&=t_j N_j|\partial_{i_j}\theta_j|_{L^2}\Big(1-\sqrt{\eps_j} \frac{|\theta _j|_{\dot{W}^{1,\infty}}}{|\partial_{i_j}\theta_j|_{L^2}}- {\frac{1}{t_jN_j}} \frac{|\partial_{i_{j-1}}\theta_{j}|_{L^2}}{|\partial_{i_j}\theta_j|_{L^2}}\Big).
\end{align*}
 Observe also that \[|\partial_{i_j}\theta_{j+1}|_{L^2}= |\partial_{i_j}\theta_j|_{L^2}.\]
Thus, 
 \[|\partial_{i_{j+1}} \theta_{j+1}|_{L^2}\ge t_j N_j|\partial_{i_j}\theta_j|_{L^2}\Big(1-\sqrt{\eps_j} \frac{|\theta _j|_{\dot{W}^{1,\infty}}}{|\partial_{i_j}\theta_j|_{L^2}}- {\frac{1}{t_jN_j}}\frac{|\partial_{i_{j-1}}\theta_{j-1}|_{L^2}}{|\partial_{i_j}\theta_j|_{L^2}}\Big),\] where $\theta_{-1}\equiv 0$. 
Define \[A_j\defeq \frac{|\partial_{i_{j}} \theta_{j}|_{L^2}}{|\partial_{i_{j+1}}\theta_{j+1}|_{L^2}}, \qquad R_{{j+1}}\defeq\frac{|\theta _{j+1}|_{\dot W^{1,\infty}}}{|\partial_{i_{j+1}}\theta_{j+1}|_{L^2}}.\]
Then, 
\begin{align*}
A_{j}&\le 2^{-{\alpha}j} \Big(1-\sqrt{\eps_j} R_j-2^{-{\alpha}j}A_{j-1}\Big)^{-1}, \\
R_{j+1}&\le R_j\Big(1+2^{-{\alpha}j}\Big)\Big(1- {\sqrt{\eps_j}}  R_j-2^{-{\alpha}j}A_{j-1})^{-1}.
\end{align*}
{Recall that} $\eps_j={\exp(-30(1 + \frac{1}{2^\alpha-1}))}2^{-2j}$ and let us bootstrap the following information: \[A_{j}\le 2\cdot 2^{-{\alpha}j}, \qquad R_j\le {\exp(10(1 + \frac{1}{2^\alpha-1}))}.\]
Then,\begin{equation}\label{eq:ineq}
\begin{split}
 A_j\le 2^{-{\alpha}j}(\frac{9}{10}-2^{-{\alpha}j}A_{j-1})^{-1}. 
\end{split}
\end{equation} 
{To start the bootstrap procedure, let us compute $A_0$ directly: {recalling} that $\theta_0 = \sin(Mx)\sin(Ly)$ and $\alpha = 1$, we have \begin{equation*}
	\begin{split}
	\theta_1(x,y) = \sin( M(x+t_1S_{\eps_1 }(N_1y) ) )\cos(Ly)
	\end{split}
	\end{equation*} and \begin{equation*}
	\begin{split}
	\rd_y \theta_1 = t_1N_1S'_{\eps_1}(N_1y)M\cos( M(x+t_1S_{\eps_1 }(N_1y) ) )\cos(Ly) - L\sin( M(x+t_1S_{\eps_1 }(N_1y) ) )\sin(Ly). 
	\end{split}
	\end{equation*} Observe that two terms in the above expression are orthogonal in $L^2(\bbT^2)$. Therefore, in this case, we have \begin{equation*}
	\begin{split}
	A_0 \le 2^{-1}(1- \sqrt{\eps_0}R_0)^{-1} < \frac{3}{5}. 
	\end{split}
	\end{equation*} Here, we also used the observation that in this case $R_0$ can be replaced by $|\rd_x\theta_0|_{L^\infty}/|\rd_x\theta_0|_{L^2}$, which is uniformly bounded. In the general case $\alpha>0$, one obtains similarly $A_0 = {A_{ J(\alpha)-1} }< \frac{3}{5}$. We omit the proof for the remaining cases of $\theta_0$, which requires only minor modifications.
Now recalling \eqref{eq:ineq} and using $A_{J(\alpha)-1} < \frac{3}{5}$ gives $A_{ {J(\alpha)} }< 1$.} Then we see that for $j\ge {J(\alpha)+1}$ we must have $A_j<2\cdot 2^{-{\alpha}j}.$
Next, let us keep that $R_j<{\exp(10(1 + \frac{1}{2^\alpha-1}))}$. We know that
\[R_{j+1}\le R_j(1+2^{-{\alpha}j})(1-\frac{1}{100}2^{-j}-2^{-{\alpha}j} A_{j-1})^{-1}.\]
On the other hand, we have that $\frac{1}{1-x}<1 {+3x}$ {for $x<2$.}
Now, \[\prod_{j=J(\alpha)}^\infty (1+2^{-{\alpha}j})<  {\exp(\frac{1}{2^\alpha-1})} \] and 
\[\prod_{j=J(\alpha)}^\infty (1-\frac{1}{100}2^{-j}-2^{-{\alpha} j} A_{j-1})^{-1}\le \prod_{j=J(\alpha)}^\infty(1+\frac{3}{100}2^{-j}+3\cdot 2^{-{\alpha}j} A_{j-1})\] since we know that $\frac{1}{100}2^{-j}+2^{-{\alpha}j} A_{j-1}\le \frac{2}{3}$ for all $j \ge J(\alpha)$.
Now, \[\prod_{j=J(\alpha)}^{\infty}(1+\frac{3}{100}2^{-j}+3\cdot 2^{-{\alpha}j} A_{j-1})\le \prod_{j=J(\alpha)}^{\infty}(1+\frac{3}{100}2^{-j}+6\cdot 2^{-2{\alpha}j}) < \exp(10(1 + \frac{1}{2^\alpha-1}))  .\]
This now concludes the proof that $R_j\le {\exp(10(1 + \frac{1}{2^\alpha-1}))}$ and $A_j\le 2\cdot 2^{-{\alpha}j}$. 
The above also shows that \[|\partial_{i_j}\theta_{j}|_{L^2}\ge {c_\alpha} 2^{\frac{{\alpha}j(j+1)}{2}}{|\rd_x\theta_0|_{L^2}}.\]
\medskip

\emph{$\dot{H}^2$ Bound:}
Finally, from the bound   
\begin{align*}
|\theta _{j+1}|_{\dot{H}^2} &\le t_j^2 N_j^2|\theta _j|_{\dot H^2} 
({ 1 + 2(t_jN_j)^{-1} + (t_jN_j)^{-2} }) +|\theta _j|_{\dot{H}^1}\frac{t_j}{\eps_j} N_j^2, \\
&\le  2^{2{\alpha}j}|\theta _j|_{{\dot{H}^2}}\Big(1+ 2 \cdot 2^{-\alpha j} + 2^{-2\alpha j} + C_\alpha \frac{|\theta _j|_{{\dot{H}^1}}}{|\theta _j|_{{\dot{H}^2}}}2^{3j}\Big).
\end{align*}
But we know that $|\theta _j|_{{\dot{H}^2}}\ge |\theta _j|_{{\dot{H}^1}}^2$ and we have a lower bound on $|\theta _j|_{{\dot{H}^1}}$. Thus, 
\[|\theta _{j+1}|_{{\dot{H}^2}}\le 2^{2{\alpha}j}|\theta _j|_{{\dot{H}^2}}\Big(1+ {2 \cdot 2^{-{\alpha}j} + 2^{-2{\alpha}j} + } {C_\alpha} 2^{-\frac{{\alpha} j(j+1)}{2}} 2^{3j}\Big).\]
Since \[\sum_{j=0}^\infty ({2\cdot 2^{-{\alpha}j} + 2^{-2{\alpha}j} + }{C_\alpha}2^{-\frac{{\alpha}j(j+1)}{2}} 2^{3j})=c_\alpha<\infty,\]
we have that 
\[|\theta _{j+1}|_{{\dot{H}^2}}\le 2^{{\alpha}j(j+1)} e^{{c_\alpha}}{|\theta_0|_{\dot{H}^2}}.\] This concludes the proof. Let us now comment on the case where $L >M$. To adapt the proof, we just need to observe that (assuming $\alpha=1$ for simplicity) at $j = 1$, we have from explicit computations that $|\theta_1|_{\dot{H}^1} \approx |\rd_y \theta_1|_{L^2}$ and \begin{equation*}
	\begin{split}
	|\theta_1|_{\dot{W}^{1,\infty}} \le 10|\theta_1|_{\dot{H}^1},\quad |\theta_0|_{\dot{H}^2} \le 10|\theta_0|_{\dot{H}^1}^2. 
	\end{split}
	\end{equation*} Therefore one can just repeat the arguments above starting with $j=1$ instead of $j=0$.
\end{proof}

\subsection{Viscous Bounds}

 To complete the proof of anomalous dissipation, we need to prove the following lemma.
\begin{lemma}\label{thm:viscous-bounds}
	Assume that $t_j, N_j, \eps_j,$ and $\theta_0$ are given as in Lemma \ref{thm:inviscid-bounds}. With $u(t)$ defined as in \eqref{eq:velocity}, the solution $\theta^{\kappa}$ to the system \eqref{visceqn} with initial data $\theta_0$ satisfies \begin{equation*}
	\begin{split}
	|\theta^{\kappa}(t)|_{H^2} \le C|\theta(t)|_{H^1}^2
	\end{split}
	\end{equation*} for some universal constant $C$ independent of $\kappa$ of $t$, where $\theta(t)$ is the inviscid solution. 
\end{lemma}
\begin{proof} {Again, for simplicity we assume that the initial data is given by $\theta_0 = \sin(Mx)\sin(Ly)$ with $M\ge L$.} Recall that $T_j=\sum_{i=0}^j t_i$ where $t_0=0$ and $t_i$ is as above. It suffices to  prove desired $H^2$ bound on the viscous solution on each time interval $[T_j, T_{j+1}]$. In particular, we prove bounds on solutions to
\[\partial_t \theta^\kappa+S_{\eps_j}(N_j y)\partial_x \theta^\kappa=\kappa\Delta \theta^\kappa,\] for $t\in [T_j, T_{j+1}],$ which is the equation when $j$ is even (when $j$ is odd the situation is almost identical).
Let us write the equations for all derivatives up to order $2$.  

\[\partial_t  \theta^\kappa+S_{\eps_j}(N_j y)\partial_x \theta^\kappa=\kappa\Delta \theta^\kappa,\]
\[\partial_t \partial_x \theta^\kappa+S_{\eps_j}(N_j y)\partial_{xx} \theta^\kappa=\kappa\Delta\partial_x \theta^\kappa,\]
\[\partial_t \partial_{xx} \theta^\kappa+S_{\eps_j}(N_j y)\partial_{xxx} \theta^\kappa=\kappa\Delta   \partial_{xx}\theta^\kappa,\]
\[\partial_t \partial_y \theta^\kappa+S_{\eps_j}(N_j y)\partial_x \partial_y \theta^\kappa+N_j S_{\eps_j}'(N_j y)\partial_x \theta^\kappa=\kappa\Delta \partial_y \theta^\kappa,\]
\[\partial_t \partial_{xy} \theta^\kappa+S_{\eps_j}(N_j y)\partial_{x}\partial_{xy} \theta^\kappa+N_j S_{\eps_j}'(N_j y)\partial_{xx} \theta^\kappa=\kappa\Delta \partial_{xy}\theta^\kappa,\]
\[\partial_t \partial_{yy} \theta^\kappa+S_{\eps_j}(N_j y)\partial_x \partial_{yy} \theta^\kappa+2N_j S_{\eps_j}'(N_j y)\partial_{xy} \theta^\kappa+N_j^2 S_{\eps_j}''(N_j y)\partial_x \theta^\kappa=\kappa\Delta \partial_{yy} \theta^\kappa.\]
In particular, from the first three equations we have that
\[|\theta^\kappa(t)|_{L^2}\le |\theta^\kappa(T_{j})|_{L^2},\qquad |\partial_x \theta^\kappa(t)|_{L^2}\le |\partial_x \theta^\kappa(T_{j})|_{L^2}, \qquad |\partial_{xx}\theta^\kappa(t)|_{L^2}\le |\partial_{xx}\theta^\kappa(T_{j})|_{L^2},\] for all $t\in [T_j, T_{j+1}]$. 
Moreover, from the second two equations (using the above bounds) we have that
\begin{gather}
  |\partial_y \theta^\kappa(t)|_{L^2}\le N_j (t-T_j)|\partial_x \theta^\kappa(T_j)|_{L^2}+|\partial_y \theta^\kappa(T_j)|_{L^2}\,,
  \\
  |\partial_{xy} \theta^\kappa(t)|_{L^2}\le N_j (t-T_j)|\partial_{xx} \theta^\kappa(T_j)|_{L^2}+|\partial_{xy} \theta^\kappa(T_j)|_{L^2}\,,
\end{gather}
for all $t\in [T_j, T_{j+1}]$. 
Finally, from the third equation and using the above bounds we get:
\[|\partial_{yy} \theta^\kappa(t)|_{L^2}\le (t-T_j)^2 N_j^2 |\partial_{xy}\theta^\kappa(T_j)|_{L^2}+(t-T_j)N_j^2\frac{1}{\eps_j}|\partial_x \theta^\kappa(T_j)|_{L^2}+|\partial_{yy} \theta^\kappa(T_j)|_{L^2},\] for all $t\in [T_j, T_{j+1}]$. 
Thus, 
\[\sup_{t\in [T_j, T_{j+1}]}|\theta^\kappa|_{\dot H^2}\le t_j^2 N_j^2 |\theta^\kappa(T_j)|_{H^2}+\frac{t_j}{\eps_j}N_j^2|\theta^\kappa(T_j)|_{H^1}+|\theta^\kappa(T_j)|_{H^2}.\]
In the above, we are using that $N_j t_j\ge 1$. 
Therefore, 
\[\sup_{t\in [T_j, T_{j+1}]}|\theta^\kappa|_{\dot H^2}\le (t_j^2 N_j^2+1) |\theta^\kappa(T_j)|_{H^2}+\frac{t_j}{\eps_j}N_j^2|\theta^\kappa(T_j)|_{H^1}.\]
Note that this bound is valid for every $j$ (even and odd). 
Now let $A_j=\sup_{j\in [T_j, T_{j+1})}|\theta^\kappa|_{H^2}$. Then, using the definition of $\eps_j, t_j,$ and $N_j$ we see:
\[\qquad A_j\le (2^{2{\alpha}j}+2)A_{j-1}+{C_\alpha} 2^{10j} \]
using the fact that $|\theta^\kappa|_{H^1}^2\le |\theta^\kappa|_{H^2}$ and the Cauchy-Schwarz inequality. {We may assume $A_0=1$ by normalizing the initial data appropriately.}
We may now define $B_j=A_j+C_\alpha 2^{10j}$ and we see:
\[B_{j}\le 2 C_\alpha 2^{10j}+(2^{2{\alpha}j}+1)(B_{j-1}-C_\alpha 2^{10(j-1)}).\]
Thus, if $j\ge {J(\alpha)}$ we have
\[B_j\le (2^{2{\alpha}j}+1)B_{j-1}\]
and it follows that \[B_j \le \tilde C_\alpha 2^{\alpha j(j+1)}.\]
In conclusion, we see that
\[\sup_{t\in [T_j, T_{j+1}]} |\theta^\kappa|_{H^2}\le \tilde C_\alpha \cdot 2^{{\alpha}j(j+1)}\] for all $j$. 
In particular, in light of the $H^1$ bound of Lemma \ref{thm:inviscid-bounds} we see that 
\[|\theta^\kappa|_{H^2}\le C_\alpha|\theta |_{H^1}^2,\] where $\theta$ is the inviscid solution. \end{proof}

\subsection{Proof of Theorem \ref{mainthm}}

The previous lemmas establish anomalous dissipation for initial data given by a pure harmonic. To conclude the proof of Theorem \ref{mainthm}, we need to treat the case of small $L^2$ perturbation, and show that the velocity field $u(t)$ we have defined in \eqref{eq:velocity} can belong to $L^1([0,1];C^{\alpha'})$ for any $\alpha'<1$. For the latter, we simply compute that \begin{equation*}
\begin{split}
|u|_{L^1([0,1];C^{\alpha'})} \lesssim \sum_{j \ge J(\alpha)} t_j N_j^{\alpha'} \lesssim \sum_{j \ge J(\alpha)} 2^{ (1+\alpha)\alpha'j -j } <+\infty 
\end{split}
\end{equation*} once we take $0<\alpha < \frac{1}{\alpha'} - 1$.

Now we assume that, for some $\varepsilon_\alpha>0$ to be determined, $\theta_0$ satisfies \begin{equation*}
\begin{split}
|\theta_0 - \lambda \psi|_{L^2} \le \varepsilon_\alpha|\theta_0|_{L^2},
\end{split}
\end{equation*}
where we may assume without loss of generality that  $\psi = \sin(Mx)$, $\lambda = 1$, and $\psi$ is orthogonal with $\theta_0-\psi$ in $L^2$. We then simply decompose
\begin{equation*}
\begin{split}
\theta_0(x,y) = \sin(Mx) + ( {\theta}_0 - \sin(Mx)) \defeq \theta_0^L + \theta_0^H
\end{split}
\end{equation*}
so that
\begin{equation*}
\begin{split}
|\theta_0|_{L^2}^2 = |\theta_0^L|_{L^2}^2 + |\theta_0^H|_{L^2}^2. 
\end{split}
\end{equation*} From the smallness assumption, we have \begin{equation*}
\begin{split}
\sqrt{1-\varepsilon_\alpha^2} |\theta_0^H|_{L^2} \le \varepsilon_\alpha|\theta_0^L|_{L^2}
\end{split}
\end{equation*} which gives \begin{equation*}
\begin{split}
|\theta_0|_{L^2} \ge |\theta_0^L|_{L^2}-|\theta_0^H|_{L^2} \ge (1 - \frac{\varepsilon_\alpha}{\sqrt{1-\varepsilon_\alpha^2}}) |\theta_0^L|_{L^2}. 
\end{split}
\end{equation*}
From previous lemmas, we have that for any $0 < \kappa \le 1$,
\begin{equation*}
\begin{split}
\frac{1}{2}( |\theta_0^L|_{L^2}^2 - |\theta^{L,\kappa}(1)|_{L^2}^2 ) = \kappa \int_0^1 |\nabla \theta^{L,\kappa}|_{L^2}^2 dt \ge \chi_\alpha |\theta_0^L|_{L^2}^2
\end{split}
\end{equation*}
where $\theta^{L,\kappa}$ is the solution to
\begin{equation*}
\begin{split}
\rd_t \theta^{L,\kappa} + u \cdot\nabla \theta^{L,\kappa} = \kappa \Delta \theta^{L,\kappa} 
\end{split}
\end{equation*}
with initial data $\theta^L_0$.
Similarly, we define $\theta^{H,\kappa}$ to be the solution with initial data $\theta^H_0$. Since the equation is linear, we have that $\theta^\kappa = \theta^{L,\kappa} + \theta^{H,\kappa}$. We now estimate that at time 1,  \begin{equation*}
\begin{split}
|\theta^\kappa(1)|_{L^2} &\le |\theta^{L,\kappa}(1)|_{L^2} + |\theta^{H,\kappa}(1)|_{L^2} 
\le (1-2\chi_\alpha)|\theta^{L}_0|_{L^2} + |\theta^{H}_0|_{L^2} \\
& \le \left( (1-2\chi_\alpha)(1 - \frac{\varepsilon_\alpha}{\sqrt{1-\varepsilon_\alpha^2}})^{-1} +  \frac{\varepsilon_\alpha}{\sqrt{1-\varepsilon_\alpha^2}} \right)  |\theta_0|_{L^2} <( 1 - \frac{1}{10}\chi_\alpha)|\theta_0|_{L^2}
\end{split}
\end{equation*} once we take $\varepsilon_\alpha = \chi_\alpha/100$, say. The proof is complete.
\hfill\qedsymbol

\subsection{Proof of Theorem \ref{mainthm2}}

We now prove Theorem \ref{mainthm2}, which establishes anomalous dissipation for arbitrary mean-zero initial data $\theta_0 \in H^2(\mathbb{T}^2)$. As in the above, we achieve it via Proposition \ref{thm:criterion1}, but with velocity vector field depending on $\theta_0$. This time, given $\alpha>0$, we take \[t_j=2^{-j},\qquad N_j=2^{{(1+\alpha)}j},\qquad \eps_j= a_0 \, {\exp\left(-30(1 + \frac{1}{2^\alpha-1})\right)}\cdot 2^{-2j},\] and define the velocity field $u(t)$ for $t \in [T_j,T_{j+1})$ by \begin{equation}\label{eq:velocity2}
\begin{split}
u(t, x,y) = \begin{cases}
\begin{pmatrix}
(-1)^{s_j}S_{\eps_j}(N_jy) \\
0
\end{pmatrix}  & j \mbox{ even},  \\
\begin{pmatrix}
0 \\
(-1)^{s_j}S_{\eps_j}(N_jx) 
\end{pmatrix} & j \mbox{ odd, }
\end{cases} 
\end{split}
\end{equation} where $T_j = \sum_{i=0}^j t_i$ with $t_0 = 0$. The constant $1 \ge  a_0>0$ and the signs $s_j \in \{0,1\} $ will be chosen depending on the initial data, as we shall see below. Apart from these additional parameters, the velocity field is exactly the same with  \eqref{eq:velocity}. 

We need to prove the assumptions of Proposition \ref{thm:criterion1}, and to do so we follow exactly the same steps from the previous section. Inspecting the proof, we see that the only place which needs to be modified is the part where we obtain an $H^1$ lower bound on the solution. In this process $s_j$ and $a_0$ will be determined. To this end we define, assuming that $\theta_j$ is given with $j$ odd, \begin{equation*}
\begin{split}
\theta_{j+1}^{\pm} = \theta_j (x \pm t_j S_{\eps_j}(N_jy),y) 
\end{split}
\end{equation*} and compute: 
\begin{equation*}
\begin{split}
\rd_{i_{j+1}} \theta_{j+1}^{\pm}(x,y) &= \rd_{i_{j+1}}\theta_j(x \pm t_j S_{\eps_j}(N_jy),y ) \pm t_jN_j S'_{\eps_j}(N_jy) \rd_{i_j}\theta_j(x + t_j S_{\eps_j}(N_jy),y ).
\end{split}
\end{equation*} A direct computation gives \begin{equation*}
\begin{split}
\sum_{\pm}|\rd_{i_{j+1}} \theta_{j+1}^{\pm} |_{L^2}^2 = 2|\rd_{i_{j+1}}\theta_j|_{L^2}^2 + 2(t_jN_j)^2 |S'_{\eps_j}(N_jy) \rd_{i_j}\theta_j|_{L^2}^2
\end{split}
\end{equation*} since the cross terms cancel each other. Then, there exists $s_j \in \{ 0,1 \}$ such that \begin{equation*}
\begin{split}
|\rd_{i_{j+1}} \theta_{j+1}^{(-1)^{s_j}} |_{L^2} \ge t_jN_j|S'_{\eps_j}(N_jy) \rd_{i_j}\theta_j|_{L^2} \ge t_jN_j |\rd_{i_j} \theta_j|_{L^2} (1 - \sqrt{\eps_j} \frac{|\theta_j|_{\dot{W}^{1,\infty}}}{|\rd_{i_j}\theta_j|_{L^2}} ).
\end{split}
\end{equation*} Therefore, in the bootstrapping scheme with $A_j$ and $R_j$, we have instead (assuming $\alpha = 1$ for simplicity) \begin{equation*}
\begin{split}
A_j \le 2^{-j} (1 - \sqrt{\eps_j}R_j)^{-1},
\end{split}
\end{equation*} while we still have the same inequality for $R_{j+1}$. We may now choose $a_0>0$ sufficiently small (depending only on $\theta_0$) to guarantee that  $A_1 <1$. Now the same bootstrap argument gives the desired lower bound.
\strut\hfill\qedsymbol

\section{Non-Uniqueness of Weak Solutions }\label{s:nonuniq}

The goal of this section is to prove Theorem~\ref{nonuniq}.
Recall that we say $\theta\in C_w(0,T; L^2(\mathbb{T}^d))$ is weak solution of the transport equation~\eqref{inveqn} on $\mathbb{T}^d\times [0,T]$ if
\begin{equation}\label{weaktrans}
  \int_0^T \int_{\mathbb{T}^d}
    \theta \paren[\big]{
      \partial_t \varphi  + u  \cdot \nabla \varphi
    }
    \rmd x \rmd t
    =  -\int_{\mathbb{T}^d} \theta_0(x) \varphi(x,0)  \rmd x
\end{equation}
holds, for all test functions $\varphi \in C_0^\infty ([0,T) \times \mathbb{T}^d)$.

  \begin{proof}[Proof of Theorem~\ref{nonuniq}]
  In what follows, we construct two distinct weak solutions; one time irreversible solution arising from a vanishing viscosity limit and one time reversible solution.
  \medskip
  
  \emph{Time irreversible weak solution:}
Let $\theta$ be a vanishing viscosity solution on $[0,2T]\times \mathbb{T}^d$ constructed as the limit of the approximating sequence $\theta^{\kappa_k}$, $\kappa_k \to 0$ of solutions of the advection diffusion equation with velocity $u_*$.
Indeed, since $\theta^{\kappa_k}$ is uniformly bounded in $L^\infty([0,2T]\times \mathbb{T}^d)$, by an application of Aubin-Lions lemma we have that $\theta^{\kappa_k} \to \theta$ in  $C([0,2T);w-L^2(\mathbb{T}^d))$ where $w-L^2$ is $L^2$ endowed with the weak topology.
Since the equation is linear and $u\in L^\infty(0,2T; L^\infty(\mathbb{T}^d))$ and $\theta\in L^\infty([0,2T]\times \mathbb{T}^d)\cap C_w([0,2T); L^2(\mathbb{T}^d))$, it is simply to verify that the weak limit $\theta$ is a weak solution of the transport equation on $[0,2T]\times \mathbb{T}^d$ in the sense of \eqref{weaktrans}.
Furthermore, since the $L^2$ norm is weakly lower semi-continuous, for all $t\in [T,2T]$ 
  \begin{equation}
  | \theta(t)|_{L^2}^2 \leq \liminf_{\kappa \to 0}  | \theta^\kappa(t)|_{L^2}^2
    \leq   | \theta_0|_{L^2}^2 - \limsup_{\kappa \to 0} \kappa \int_0^T |\nabla \theta^\kappa|_{L^2}^2 \rmd t\leq  (1- \chi_\alpha )| \theta_0|_{L^2}^2 <| \theta_0|_{L^2}^2 
  \end{equation}
upon applying Theorem \ref{mainthm2}.
Thus,
\begin{equation}
\label{dissipating}
  \sup_{t\in [T, 2T]} | \theta(t)|_{L^2}^2 \leq (1 - \chi_\alpha) \abs{\theta_0}_{L^2}^2 <    | \theta_0 |_{L^2}^2 \,,
\end{equation}
and the inviscid solution has lost a non-zero fraction of its initial energy after time $T$.
\medskip
  
 \emph{Time reversible weak solution:}
  We now construct another weak solution $\bar{\theta}$ distinct from $\theta$ on the interval $[T,2T]$.  This solution is defined by the formula
    \begin{equation}
 \bar \theta(t) = 
  \begin{cases} \theta(t) & t\in [0,T),\\
  \theta(2T-t) & t\in [T,2T].
  \end{cases}
  \end{equation}
  Note first that, since $\theta\in C_w([0,T]; \mathbb{T}^d)$, by construction $\bar\theta\in C_w([0,2T]; \mathbb{T}^d)$.
   We now check that it is a weak solution on the entire time interval with the velocity $u$.   That is, we aim to show that for any $\varphi \in C^\infty_0([0, 2T)\times \mathbb{T}^d)$ we have that
   \begin{equation}
\int_0^{2T} \int_{\mathbb{T}^d}  \bar\theta (x,t) \partial_t \varphi(x,t)  \rmd x \rmd t+ \int_0^{2T} \int_{\mathbb{T}^d} \bar\theta (x,t)u(x,t)  \cdot \nabla \varphi(x,t)  \rmd x \rmd t =  -\int_{\mathbb{T}^d} \theta_0 (x)  \varphi(x,0)  \rmd x .
\end{equation}
To proceed, divide $\varphi$ into even and odd parts about $t=T$:
\begin{equation}
\varphi = \varphi_{e} + \varphi_{o}, \qquad \varphi_{e/o}\defeq \frac{ \varphi(T + (t-T)) \pm  \varphi(T - (t-T))}{2}.
\end{equation}
Note that since $\theta (x,t)$ is even and $u(x,t)$ is odd about $t=T$, the left-hand-side of the above expression vanishes identically for the even part of $\varphi$, namely it reduces to 
   \begin{equation}
\int_0^{2T} \int_{\mathbb{T}^d}  \bar\theta (x,t) \partial_t \varphi_o(x,t)  \rmd x \rmd t+ \int_0^{2T} \int_{\mathbb{T}^d} \bar\theta (x,t)u(x,t)  \cdot \nabla \varphi_o(x,t)  \rmd x \rmd t =  -\int_{\mathbb{T}^d} \theta_0 (x)  \varphi(x,0)  \rmd x .
\end{equation}
Splitting up different regions, we have
   \begin{align}\nonumber
   \int_0^{T} \int_{\mathbb{T}^d}& \theta (x,t) \partial_t \varphi_o(x,t)  \rmd x \rmd t+ \int_0^{T} \int_{\mathbb{T}^d} \theta (x,t)u_*(x,t)  \cdot \nabla \varphi_o(x,t)  \rmd x \rmd t =  -\int_{\mathbb{T}^d} \theta_0 (x)  \varphi(x,0)  \rmd x\\ \nonumber
      &-\int_{T}^{2T}  \int_{\mathbb{T}^d} \theta (x,2T-t) \partial_t \varphi_o(x,t)  \rmd x \rmd t+ \int_{T}^{2T} \int_{\mathbb{T}^d} \theta (x,2T-t)u_*(x,2T-t) \cdot \nabla \varphi_o(x,t)  \rmd x \rmd t.
   \end{align}
Changing variables and introducing $\psi(x,\tau) = \varphi_o(x, 2T-\tau )\in C^\infty_0([0, 2T]\times \mathbb{T}^d)$ and additionally $\psi(x,T)=0$ since $\varphi_o$ vanishes at $t=T$ owing to the fact that it is odd,  we have
   \begin{align}\nonumber
   \int_0^{T} \int_{\mathbb{T}^d}& \theta (x,t) \partial_t \varphi_o(x,t)  \rmd x \rmd t+ \int_0^{T} \int_{\mathbb{T}^d} \theta (x,t)u_*(x,t)  \cdot \nabla \varphi_o(x,t)  \rmd x \rmd t =  -\int_{\mathbb{T}^d} \theta_0 (x)  \varphi(x,0)  \rmd x\\ \nonumber
      &+  \int_0^{T} \int_{\mathbb{T}^d} \theta (x,\tau) \partial_\tau \psi(x,\tau )  \rmd x \rmd \tau+    \int_0^{T} \int_{\mathbb{T}^d} \theta (x,\tau)u_*(x,\tau) \cdot \nabla \psi(x,\tau )  \rmd x \rmd \tau .
   \end{align}
Since $\theta$ is a weak solution in the class $C_w([0,T]; \mathbb{T}^d)$ on the interval $[0,T]$ and $\psi(x,T)=0$ while $\psi(x,0)=\phi_0(x,2T)$,
      \begin{align}\nonumber
  \int_0^{T} \int_{\mathbb{T}^d}\left[ \theta (x,\tau) \partial_\tau \psi(x,\tau ) + \theta (x,\tau)u_*(x,\tau) \cdot \nabla \psi(x,\tau ) \right] \rmd x \rmd t  =-\int_{\mathbb{T}^d} \theta_0 (x)  \varphi_o(x,2T)  \rmd x.
   \end{align}
We also have
         \begin{align}\nonumber
 \int_0^{T} \int_{\mathbb{T}^d}\left[ \theta (x,t) \partial_t \varphi_o(x,t) + \theta (x,t)u_*(x,t)  \cdot \nabla \varphi_o(x,t) \right] \rmd x \rmd t =  -\int_{\mathbb{T}^d} \theta_0 (x)  \varphi_o(x,0)  \rmd x.
   \end{align}
   Since
   \begin{equation}
    \varphi_o(x,0)=  -\varphi_o(x,2T) = \frac{1}{2}\varphi(x,0),
   \end{equation}
we find that $\bar{\theta}\in C_w([0,2T]; \mathbb{T}^d)$ is a weak solution on the interval $[0,2T]$.       Finally, we note that
     \begin{equation}
 | \bar{\theta}(2T)|_{L^2}^2 =  | \theta_0|_{L^2}^2.
     \end{equation}
In light of \eqref{dissipating}, the solutions $\theta$ and $\bar{\theta}$ are distinct. 
  \end{proof}

\section{Discussion:  Obukhov-Corrsin Theory}\label{disc}

In the context of passive scalar turbulence, Obukhov \cite{Obukhov49} and Corrsin \cite{Corrsin51} studied the `inertial-range' scaling behavior of scalar structure functions $S_p^\theta (\ell)\defeq \langle |\delta_\ell \theta|^p\rangle\sim \ell^{\zeta_p(\theta)}$ 
in a fully developed homogenous isotropic velocity field exhibiting Kolmogorov 1941 (K41) `monofractal' scaling \cite{Kolmogorov41}
\begin{equation}
S_p^u(\ell)\defeq \langle |\delta_\ell u|^p\rangle \sim (\ve \ell)^{p/3}, \qquad \ell_\nu \ll \ell \ll L^u
\end{equation}
for all $p\ge 1$  where $L^u$ is the integral scale of the velocity field and $\ell_\nu$ is the dissipation scale (the K41 prediction being $\ell_\nu =(\nu^3/\varepsilon)^{1/4}$ where $\varepsilon\defeq \lim_{\nu \to 0} \nu\langle |\nabla u^\nu|^2\rangle >0$ is the anomalous energy dissipation rate).  Said another way, in the idealized limit $\nu \to 0$, the velocity field  is assumed to be $1/3$--H\"{o}lder and not better.
Based on dimensional grounds,  Obukhov and Corrsin  independently predicted that the scalar field would also exhibit the same scaling 
\begin{equation}
S_p^\theta (\ell)\defeq \langle |\delta_\ell \theta|^p\rangle \sim  (\chi/\varepsilon^{1/3})^{p/2} \ell^{p/3}, \qquad \ell_\nu \lesssim \ell_\kappa \ll \ell \ll L^\theta\lesssim L^u
\end{equation}
where $\chi\defeq \lim_{\kappa,\nu \to 0} \kappa\langle |\nabla \theta^\kappa|^2\rangle>0$ is the (presumed) anomalous dissipation of the passive scalar, $ L^\theta$ is the typical length-scale of the scalar input initially or by a force, and $\ell_\kappa$ is the dissipative length for the scalar field ($\ell_\kappa = (\kappa^3/\ve)^{1/4}$ in the Corrsin-Obukhov theory).  Their scaling theory can be generalized  as
\begin{equation}
S_p^u(\ell)\sim \ell^{\alpha p}, \quad \alpha \in (0,1) \qquad \text{implies} \qquad S_p^\theta(\ell)\sim \ell^{\left(\frac{1-\alpha}{2}\right) p}.
\end{equation}
In the idealized limit of $\nu, \kappa\to 0$, this says that if the velocity $u\in C^\alpha$ is H\"{o}lder with exponent $\alpha\in (0,1)$ and not better, then the scalar should be H\"{o}lder $\theta\in C^\beta$ with exponent $\beta= (1-\alpha)/2$ and not better. These constraints can be understood as a consequence of the fractal geometry of scalar level sets in rough velocities \cite{ConstantinProcaccia93,ConstantinProcaccia94}.  Moreover,  the entire picture has been generalized to accommodate (the more realistic setting) of multifractal velocity fields with the property that $S_p^u(\ell)\sim  \ell^{\zeta_p(u)}$ where $\zeta_p(u)$ may depend non-linearly on $p$ resulting in constraints on the multifractal spectrum of the scalar $\zeta_p(\theta)$ \cite{Eyink96}.

In analogy to the Onsager conjecture for the dissipation anomaly of kinetic energy in incompressible fluids \cite{Onsager49}, one can regard the above theory as setting a threshold condition for the anomalous dissipation of scalar energy \cite{Eyink96}.  Namely, if  $u\in C^\alpha$  and  $\theta^\kappa \in C^\beta$ uniformly  then
\begin{equation}
\chi\defeq  \kappa \int_0^T \int_{\mathbb{T}^d} |\nabla \theta^\kappa|^2 \rmd x \rmd t \to 0 \qquad \text{unless} \qquad \beta>\frac{1-\alpha}{2}. 
\end{equation}
Along these lines, we first establish an upper bound on the dissipation for vanishing diffusion limits in rough velocity fields.  A similar estimate was provided for viscous energy dissipation in the context of Onsager's conjecture for hydrodynamic turbulence \cite{DrivasEyink19}.  We also study what happens when the velocity field is smooth up until a single point in time where it may lose regularity.  The latter is relevant to the problem in which an inertial range for the velocity field evolves dynamically by some cascade process to the point where the field becomes non-smooth in a way consistent with the observed long-time inertial range scaling in real turbulence.  In fact, one has the following result.
\begin{theorem}\label{thm1}
Let $u\in L^1([0,T]; C^\alpha(\mathbb{T}^d))$ for $\alpha\in (0,1]$ be a given divergence free vector field.  Suppose that the family $\{\theta^\kappa\}_{\kappa>0}$ is bounded in $L^\infty([0,T] ; C^\beta(\mathbb{T}^d))$ for $\beta\in (0,1]$ uniformly in $\kappa$, then 
\begin{equation}\label{noanombnd}
 \kappa \int_0^T \int_{\mathbb{T}^d} |\nabla \theta^\kappa|^2 \rmd x \rmd t\le C  \kappa^{\frac{\alpha +2\beta-1}{\alpha+1}}
\end{equation}
for an absolute constant $ C$ depending only on $T$ and the H\"{o}lder norms of the solutions.
In particular, if $\beta>(1-\alpha)/2$ then there can be no anomalous scalar dissipation.  
If furthermore
\begin{equation*}
u\in  L^1_{loc}([0,T); W^{1,\infty}(\mathbb{T}^d))\cap L^1([0,T]; C^\alpha(\mathbb{T}^d))
\qquad
\text{for } \alpha<1\,,
\end{equation*}
 and if $\beta=(1-\alpha)/2$, then 
\begin{equation}
\lim_{\kappa\to 0} \kappa \int_0^T \int_{\mathbb{T}^d} |\nabla \theta^\kappa|^2 \rmd x \rmd t=0.
\end{equation}
\end{theorem}

\begin{proof}
Let $\ol{f}_\ell = \varphi_\ell *f$ for any $\ell>0$.  Mollifying the equations, one finds
\begin{equation}
\partial_t \ol{(\theta^\kappa)}_\ell + \ol{u}_\ell \cdot \nabla \ol{(\theta^\kappa)}_\ell = \kappa \Delta \ol{(\theta^\kappa)}_\ell - \nabla \cdot \tau_\ell(u,\theta^\kappa)
\end{equation}
where $\tau_\ell(f,g)= \ol{(fg)}_\ell - \ol{f}_\ell \ol{g}_\ell$.  A straightforward calculation  for any $0\le t\le T$ shows that
\begin{align}\nonumber
 \kappa \int_{t}^T \int_{\mathbb{T}^d} |\nabla \theta^\kappa|^2 \rmd x \rmd t' &= \int_t^T \int_{\mathbb{T}^d} \nabla\ol{(\theta^\kappa)}_\ell\cdot \tau_\ell(u,\theta^\kappa) \rmd x \rmd t' +  \kappa \int_t^T\int_{\mathbb{T}^d} |\nabla \ol{(\theta^\kappa)}_\ell|^2 \rmd x \rmd t' \\
 &\quad + \frac{1}{2} \int_{\mathbb{T}^d}  \tau_\ell(\theta^\kappa(t),\theta^\kappa (t)) \rmd x-  \frac{1}{2} \int_{\mathbb{T}^d}  \tau_\ell(\theta^\kappa(T),\theta^\kappa (T )) \rmd x.
\end{align}
Using standard estimates for mollified gradients and the Constantin-E-Titi \cite{ConstantinEEA94} commutator estimate 
\begin{equation}
|\nabla \ol{f}_\ell|_{L^\infty}\le  |\theta |_{C^\alpha} \ell^{\alpha-1}, \qquad |\tau_\ell (f,g)|_{L^\infty} \le |\theta |_{C^\alpha} |g|_{C^\beta} \ell^{\alpha+\beta}, \qquad f\in C^\alpha, \ g\in C^\beta
\end{equation}
together with the fact that $\tau_\ell(f,f)\ge 0$ we arrive at an upper bound for the scalar dissipation
\begin{equation}
 \kappa \int_{t}^T \int_{\mathbb{T}^d} |\nabla \theta^\kappa|^2 \rmd x \rmd t' \le    \ell^{\alpha+2\beta-1} |\theta^\kappa|_{L^\infty_tC^\beta_x}^2  |u|_{L^1(t,T; C^\alpha_x)}  +  \left(\kappa (T-t) \ell^{2(\beta-1)} + \ell^{2\beta}\right) |\theta^\kappa|_{L^\infty_tC^\beta_x}^2.
\end{equation}
Setting $t=0$ and optimizing $\ell$ as a function of $\kappa$ we find $\ell= \kappa^{1/(\alpha+1)}$ and \eqref{noanombnd} follows.  The second statement of the theorem follows by dividing the time interval  into $[0, T-\epsilon] \times [T-\epsilon, T]$ and using the assumed $C^1$ regularity of $u$ on the interval $[0, T-\epsilon] $ together with the uniform H\"{o}lder on the entire interval $[0,T]$ and the fact that $\epsilon$ is allowed arbitrarily small (and can vanish as $\kappa\to 0$).
\end{proof}

In light of Theorems \ref{mainthm} and \ref{thm1}, we conclude with an open question.
\begin{question}
  Fix $\alpha\in (0,1)$.
  Does there exist divergence-free vector field
  \begin{equation*}
    u\in   L^1([0,1]; C^\alpha(\mathbb{T}^d))
  \end{equation*} such that $\{\theta^\kappa\}_{\kappa>0}$ is bounded in $L^\infty([0,T] ; C^\beta(\mathbb{T}^d))$ for every $\beta<(1-\alpha)/2$,and
\begin{equation*}
\liminf_{\kappa\to 0} \kappa \int_0^T \int_{\mathbb{T}^d} |\nabla \theta^\kappa|^2 \rmd x \rmd t>0 \ ?
\end{equation*}
\end{question}

Finally, we comment briefly on the nonlinear problem: establishing anomalous dissipation for solutions of Navier-Stokes equations
\begin{align*}
\partial_t u^\nu + u^\nu \cdot \nabla u^\nu &= -\nabla p^\nu + \nu \Delta u^\nu,\\
\nabla \cdot u^\nu &= 0.
\end{align*}
As for passive scalars, experimental and numerical observations of hydrodynamic turbulence suggest that kinetic energy dissipation is non-vanishing in the limit of zero viscosity \cite{Sreenivasan84,Sreenivasan98,PearsonKrogstadEA02,KanedaIshiharaEA03}, i.e.\ there exists $\ve>0$ independent of $\nu$ such that, in turbulent regimes, a family of Leray-Hopf solutions $\{u^\nu \}_{\nu>0}$ satisfies
\begin{equation}\label{anomalousdissNS}
\nu\int_0^T  \int  \!  |\nabla u^\nu(x,t)|^2 \rmd x \rmd t \geq \varepsilon >0.
\end{equation}
This phenomenon of anomalous dissipation is so fundamental to our modern understanding of turbulence that it is often termed the "zeroth law".  In 1949 \cite{Onsager49}, Lars Onsager offered significant insight 
into this phenomena in asserting that  it requires that, at high Reynolds number, flow develop  structures approximating singular ones with
H\"{o}lder exponents not exceeding $1/3$.  This assertion has since been proved \cite{Eyink94,ConstantinEEA94} and dissipative weak solutions of the Euler equations with lower regularity have been constructed in a series of works using  convex integration  \cite{DeLellisSzekelyhidi09,DeLellisSzekelyhidi10,DeLellisSzekelyhidi12,Isett18,BuckmasterDeLellisEA19} and culminating in a construction of non-conservative solutions in the class $C_t C^{1/3-}_x$ by P. Isett.
However, to this day, none of these constructions are achieved as zero viscosity limits of  Navier-Stokes solutions obeying a physical energy balance (e.g. Leray-Hopf weak solutions).  In the present paper, we solved an  analogous problem for passive scalars in a setting which models the effect of a finite-time singularity in an inviscid problem
on anomalous dissipation in the corresponding viscous problem.  Our result follows from a sufficient condition for anomalous dissipation assuming that the inviscid solution becomes singular in
a controlled way. It is possible that one could deduce anomalous dissipation in the vanishing viscosity limit of Navier Stokes solutions under some conditions on a (hypothetical) blowup in the Euler equation.

\newcommand{\etalchar}[1]{$^{#1}$}

\end{document}